\documentclass[12pt]{article}
\usepackage{amsthm}
\usepackage{amssymb}
\usepackage{amsfonts}
\usepackage{amsmath}
\usepackage{xcolor}

\begin{document}

\newtheorem{theorem}{Theorem}[section]
\newtheorem{defi}{Definition}[section]
\newtheorem{corollary}[theorem]{Corollary}
\newtheorem{observation}[theorem]{Observation}
\newtheorem{lemma}[theorem]{Lemma}
\newtheorem{proposition}[theorem]{Proposition}
\newtheorem{conjecture}[theorem]{Conjecture}
\newtheorem{step}[theorem]{Step}
\newtheorem{example}[theorem]{Example}
\newtheorem{remark}[theorem]{Remark}

\newcommand{\tr}{{\rm tr}}
\newcommand{\supp}{{\rm supp}}
\newcommand{\linspan}{{\rm span}}
\newcommand{\rank}{{\rm rank}}
\newcommand{\diag}{{\rm Diag}}
\newcommand{\Image}{{\rm Im}}
\newcommand{\Ker}{{\rm Ker}}
\newcommand\Bdd{{\cal B}}
\newcommand\Effect{{\cal E}}
\newcommand\Proj{{\cal P}}
\newcommand\Sca{\mathcal{SC}}
\newcommand\Borel{{\cal B}}
\newcommand\calC{{\cal C}}
\newcommand\calT{{\cal T}}
\newcommand\FiniteRank{{\cal F}}
\newcommand\R{\mathbb{R}}
\newcommand\N{\mathbb{N}}
\newcommand\C{\mathbb{C}}

\newcommand\E{\ell}
\newcommand\calM{{\cal M}}
\newcommand\pc{\mathfrak{c}}

\newcommand{\enp}{\begin{flushright} $\Box$ \end{flushright}}

\newcommand{\Pn}{P^1_n(\mathbb C)}
\newcommand{\Mre}{{\mathcal M}_{r}^1}
\newcommand{\Mrk}{{\mathcal M}_{r}^2}
\newcommand{\LMre}{{\mathcal{LM}}_{r}^1}
\newcommand{\LMrk}{{\mathcal{LM}}_{r}^2}
\newcommand{\Mu}{{\mathcal M}_{u}}
\newcommand{\LMu}{{\mathcal{LM}}_{u}}
\newcommand{\Eu}{\mathcal{E}_s}

\newcommand{\Pne}{I_n^1(\mathbb F)}
\newcommand{\Se}{\mathbb S^1}

\title{Local automorphisms of some classical groups\thanks{This research was supported by 
by the National Research, Development and Innovation Office of
Hungary, NKFIH, Grant No. K134944, $\text{ADVANCED}\_150059$, and $\text{EXCELLENCE}\_151232$ and
the grants N1-0368, J1-60025, and P1-0288 from ARIS, Slovenia.}}

\author{Lajos Moln\'ar\footnote{Bolyai Institute, University of Szeged, and 
HUN-REN–SZTE Analysis and Applications Research Group,
Aradi v\'ertan\'uk tere 1., H-6720 Szeged,
Hungary, and HUN-REN Alfréd Rényi Institute of Mathematics, Reáltanoda u. 13-15., Budapest H-1053, Hungary, and 
Department of Analysis and Operations Research, Institute of Mathematics, Budapest University of Technology and Economics, Műegyetem rkp. 3., H-1111 Budapest, Hungary,
molnarl@math.u-szeged.hu}
\enskip and
Peter \v Semrl\footnote{ Institute of Mathematics, Physics and Mechanics, Jadranska 19, SI-1000 Ljubljana, Slovenia; Faculty of Mathematics and Physics, University of Ljubljana, Jadranska 19, SI-1000 Ljubljana, Slovenia, peter.semrl@fmf.uni-lj.si}}

\date{}

\maketitle

\begin{abstract}
A map on a group into itself is called a local automorphism if at any two points of the group, it can be interpolated by an automorphism of that group.
In this paper we investigate the question of how local automorphisms of some classical groups are related to automorphisms. In some cases it turns out that the local automorphisms are in fact automorphisms. In the remaining cases we show that the local automorphisms are still closely related to the automorphisms.	
\end{abstract}
\maketitle





\section{Introduction}

The concept of local maps (such as local automorphisms, local derivations) on Banach algebras was introduced independently by Kadison \cite{Kad}, and Larson and Sourour \cite{LaS}.
A linear map $\phi$ of a Banach algebra ${\cal A}$ is a local automorphism if for every $a\in {\cal A}$, there exists an automorphism  $\phi_a : {\cal A} \to {\cal A}$ (depending on $a$)
such that $\phi (a) = \phi_a (a)$. The main question is whether local implies global. More precisely, usually we have two questions. Is every local automorphism an automorphism? When the answer is negative, one may ask whether, imposing the extra assumption that the local automorphism $\phi$ under consideration is bijective, it is necessarily an automorphism (cf. \cite{LaS})? 

An automorphism $\phi$ of an algebra ${\cal A}$ is a bijective map having two algebraic properties, linearity and multiplicativity. When defining a local automorphism as above, we actually keep the linearity assumption and replace the multiplicativity assumption by the local property. It 
seems natural to ask whether we can omit also the linearity assumption. If we just drop it, then the same question as above makes no sense for most Banach algebras. For example, if for a given unital Banach algebra ${\cal A}$ all automorphisms are inner, then it is trivial that a (not necessarily linear) map $\phi : {\cal A} \to {\cal A}$ is a local automorphism if and only if for every $a \in {\cal A}$, the element $\phi (a)$ is contained in the similarity orbit of $a$, nothing more can be said.  Therefore, a stronger assumption is needed when we want to study local automorphisms in the absence of the linearity assumption. This led the second author of this paper to introduce the notion of 2-local automorphisms.  In \cite{Sem}, a not necessarily linear map $\phi : {\cal A} \to {\cal A}$ on a Banach algebra ${\cal A}$ was defined to be a 2-local automorphism if for every pair of elements $a,b \in {\cal A}$,  there exists an automorphism $\phi_{a,b} : {\cal A} \to {\cal A}$ 
such that $\phi (a) = \phi_{a,b} (a)$ and  $\phi (b) = \phi_{a,b} (b)$.

Since then quite a lot of attention has been paid to local and 2-local isomorphisms, derivations, Lie homomorphisms, isometries, etc. For a gentle introduction to the topic, one can consult the survey paper \cite{AKP} or the third chapter of the book \cite{Mol}. At the beginning, this problem was studied mainly from the functional analysis point of view. Recently, there have been more results of this type that belong to abstract algebra. 

However, it is somewhat surprising that while local automorphisms of algebras have been studied a lot, the problem has been almost untouched in the setting of groups. On the one hand, it is true that when dealing with groups, the original definition of Kadison, Larson, and Sourour makes no sense because there the linearity is involved. However, the notion of 2-local automorphisms seems to be very natural in the setting of groups. 

Our attention in this paper is limited to groups entirely, therefore we adopt the following definition of local automorphisms.

\begin{defi}
Let $G$ be a group. A map $\phi : G \to G$ is called a local automorphism of $G$ if for every pair of elements $a, b\in G$, there exists an automorphism  $\phi_{a,b} : G \to G$ 
such that $\phi (a) = \phi_{a,b} (a)$ and $\phi (b) = \phi_{a,b} (b)$.
\end{defi}

Of course, the question is whether every local automorphism is necessarily an automorphism. When the answer happens to be negative, we can try to describe the general form of the local automorphisms, and then see whether every bijective local automorphism is an automorphism. 
We are aware of essentially one paper only that is devoted to this problem concerning groups. In Theorem 3.2 in \cite{MoS}, it was proved that every local automorphism of the general linear group on a complex infinite-dimensional separable Hilbert space $H$ is an automorphism. The same is true for the unitary group on $H$. Indeed, to see that, one should combine Theorem 2.2 in \cite{MoS} with Theorem 2.5 in \cite{ML12b}. 

Below we study those questions for some of the classical groups over the real or complex field. More precisely, we deal with the general linear group, the special linear group, the unitary group and the special unitary group. We prove that "local implies global" holds in some of the cases. And even when local automorphisms are not necessarily automorphisms, we show that they can still be obtained from automorphisms by some rather minor modifications. As a matter of fact, we have the first type of answers, i.e., local implies global, for the real special linear group and for the complex special unitary group, while concerning the remaining groups we have the second type of answers, i.e., local implies global modulo some minor modifications. We mention that, as we will see, the finite-dimensional setting that we are considering here is more complicated than the infinite-dimensional one considered in \cite{MoS} (compare the present results with the ones in \cite{MoS}). The main reasons are the nonexistence of nontrivial multiplicative scalar functions on the mentioned classical groups defined on infinite-dimensional Hilbert spaces
and the automatic continuity results for their group automorphisms that do not hold in the finite-dimensional case.

Before one may think that local automorphisms of groups are usually close to automorphisms, let us present a simple example illustrating that local automorphisms in general can  behave rather widely, in fact they can be very far from automorphisms. Consider the additive group $(\mathbb{R}, +)$ of the reals.
The automorphisms $\phi : \mathbb{R} \to \mathbb{R}$ are the bijective solutions of the Cauchy functional equation
$$
\phi (x+y) = \phi (x) + \phi (y), \ \ \ x,y \in \mathbb{R}.
$$
It is well-known (and easy to see) that $\phi : \mathbb{R} \to \mathbb{R}$ satisfies the Cauchy functional equation if and only if $\phi$ is a linear map on $ \mathbb{R}$ into itself, where $ \mathbb{R}$ is considered as a vector space over the field $\mathbb{Q}$ of rational numbers.

So, assume that $\phi$ is a local automorphism of the additive group $\mathbb{R}$. Then clearly, $\phi (0) = 0$. If $a,b \in \mathbb{R}$ are linearly independent over $\mathbb{Q}$ then $\phi (a) = \phi_{a,b} (a)$ and $\phi (b) = \phi_{a,b} (b)$ are linearly independent, too. If $a,b \in \mathbb{R}$ are linearly dependent and both nonzero, then $a = \lambda b$ for some nonzero $\lambda \in \mathbb{Q}$, and consequently, $$
\phi (a) = \phi_{a,b} (a) = \phi_{a,b} (\lambda b) = \lambda \phi_{a,b} (b) = \lambda \phi (b).
$$

For any nonzero $a \in \mathbb{R}$, we denote by $[a]$ the one-dimensional $\mathbb{Q}$-linear subspace of $\mathbb{R}$ spanned by $a$, i.e., $[a] = \{ \lambda a \, | \, \lambda \in \mathbb{Q} \}$. The symbol $\mathbb{P}_{\mathbb{Q}} (\mathbb{R})$ stands for the projective space over the $\mathbb{Q}$-linear space $\mathbb{R}$, $\mathbb{P}_{\mathbb{Q}} (\mathbb{R}) = \{ [a] \, | \, a \in \mathbb{R} \setminus \{ 0 \} \}$. The above considerations yield that the local automorphism $\phi$ induces an injective map $\xi : \mathbb{P}_{\mathbb{Q}} (\mathbb{R}) \to \mathbb{P}_{\mathbb{Q}} (\mathbb{R})$ such that  for every nonzero $a \in \mathbb{R}$, the local automorphism $\phi$ maps $[a]$ bijectively onto $\xi ([a])$ (and if $\phi (a) = b \in \xi ([a])$ then $\phi (\lambda a) = \lambda b$, $\lambda \in \mathbb{Q}$).

Conversely, let $\xi : \mathbb{P}_{\mathbb{Q}} (\mathbb{R}) \to \mathbb{P}_{\mathbb{Q}} (\mathbb{R})$ be any injective map. For every element $\alpha \in  \mathbb{P}_{\mathbb{Q}} (\mathbb{R})$ we chose nonzero $a_\alpha \in \alpha$ and $b_\alpha \in \xi (\alpha)$. We define a map $\phi : \mathbb{R} \to \mathbb{R}$ in the following way. First, set $\phi (0)= 0$. For every nonzero $x \in \mathbb{R}$ there exist a unique $\alpha \in  \mathbb{P}_{\mathbb{Q}} (\mathbb{R})$ and a unique rational number $\lambda$ such that $x = \lambda a_\alpha$. Then we define $\phi (x) = \lambda b_\alpha$. Such a map $\phi$ is a local automorphism of the additive group $\mathbb{R}$. Indeed, all we need to verify is that for any two pairs $x,y \in \mathbb{R}$ and $u,v \in \mathbb{R}$ of linearly independent vectors over $\mathbb{Q}$, there exists a bijective function $\phi : \mathbb{R} \to \mathbb{R}$ which is a solution of the Cauchy functional equation and satisfies $\phi (x) = u$ and $\phi (y) = v$. This follows easily from the fact that for every pair $x,y \in \mathbb{R}$, that is linearly independent over $\mathbb{Q}$, there exists a Hamel basis of $\mathbb{R}$ containing $x$ and $y$. Therefore, the local automorphisms of the additive group $\mathbb{R}$ correspond exactly to the injective self-maps of the projective space $\mathbb{P}_{\mathbb{Q}}(\mathbb R)$ which can apparently behave quite arbitrarily.

We close the introduction with the following.
In various areas of mathematics, a lot of special types of maps are considered and quite often the question whether local properties imply global nice behaviour is studied in various contexts. We recall from the first two paragraphs  of this section that locality as described in this paper has been first introduced in operator theory and functional analysis relating Banach algebras. After introducing linear $1$-local maps, the next step was to introduce 2-local maps (without assuming their linearity). However, such maps can be studied not only on Banach algebras or more abstract algebras but also on various other algebraic structures such as groups, lattices, etc., which are equipped with one operation or relation only. Then, as in the present paper, we can simply speak of local maps rather than of 2-local maps and investigate the question of when local implies global. 

We believe that the results presented in this paper will be interesting for  mathematicians working in various areas. Therefore, we try to make the paper essentially self-contained and hence accessible to researchers not specialized in classical groups.

\section{The automorphisms of the considered classical groups}

In this section we present the descriptions of the automorphisms of the  matrix groups whose local automorphisms we are studying in the paper. For fundamental results of the theory of the  automorphisms of classical groups we refer to \cite{Die}.
 
First, we fix some notation. Throughout the paper we assume that $n$ is a positive integer with $n\geq 3$.
We will restrict our attention to the automorphisms of some matrix groups over the field $\mathbb{F} \in \{ \mathbb{R} , \mathbb{C} \}$. We set $\mathbb{F}^\ast = \mathbb{F} \setminus \{ 0 \}$. 
We denote by $M_n(\mathbb F)$ the algebra of all $n\times n$ matrices with entries in $\mathbb F$.
The symbol $GL_n (\mathbb{F})$ denotes the general linear group consisting of all invertible elements of $M_n(\mathbb F)$. 

Let $\sigma$ be an automorphism of the field $\mathbb{F}$ (field automorphism for short) and $A = [a_{ij}] \in M_n (\mathbb{F})$. Then the symbols $A^t$ and  $A_\sigma$ stand for the transpose of $A$ and the matrix obtained from $A$ by applying $\sigma$ entrywise, $[a_{ij}]_\sigma = [\sigma (a_{ij})]$, respectively. Similarly, if $x\in \mathbb{F}^n$, then $x_\sigma$ denotes the vector obtained from $x$ by applying $\sigma$ entrywise to its components.

The next two propositions concern the structures of the automorphisms of $GL_n (\mathbb{F})$. The first one is about the real case.

\begin{proposition}\label{P:GLaut}
The map $\phi$ of the group $GL_n (\mathbb{R})$ is an automorphism if and only if there exist a $T \in GL_n (\mathbb{R})$
and a multiplicative function $g : \mathbb{R}^\ast \to  \mathbb{R}^\ast$ such that either 
\begin{itemize}
\item the function  $f : \mathbb{R}^\ast \to  \mathbb{R}^\ast$ defined by $f(\lambda) = g(\lambda)^n \lambda $, $\lambda \in \mathbb{R}^\ast$  is an automorphism of the multiplicative group $\mathbb{R}^\ast$ and
\begin{equation}\label{jao}
\phi (A) = g(\det A) TA T^{-1} , \ \ \ A \in  GL_n (\mathbb{R}),
\end{equation}
or
\item the function  $f : \mathbb{R}^\ast \to  \mathbb{R}^\ast$ defined by $f(\lambda) = g(\lambda)^n \lambda^{-1}$, $\lambda \in \mathbb{R}^\ast$ is an automorphism of the multiplicative group $\mathbb{R}^\ast$ and
\begin{equation}\label{jaojao}
\phi (A) = g(\det A) T(A^{-1})^t T^{-1} , \ \ \ A \in  GL_n (\mathbb{R}).
\end{equation}
\end{itemize}
\end{proposition}

The complex case is somewhat more complicated.

\begin{proposition}\label{P:GLautcomplex}
The map $\phi$ of the group $GL_n (\mathbb{C})$ is an automorphism if and only if there exist a $T \in GL_n (\mathbb{C})$, an automorphism $\sigma$ of the field $\mathbb C$, and a multiplicative function $g : \mathbb{C}^\ast \to  \mathbb{C}^\ast$ such that either
\begin{itemize}
\item the function  $f : \mathbb{C}^\ast \to  \mathbb{C}^\ast$ defined by $f(\lambda) = g(\lambda)^n \sigma (\lambda )$, $\lambda \in \mathbb{C}^\ast$ is an automorphism of the multiplicative group $\mathbb{C}^\ast$ and
\begin{equation}\label{E:glcomplex1}
\phi (A) = g(\det A) TA_\sigma T^{-1} , \ \ \ A \in  GL_n (\mathbb{C}),
\end{equation}
or
\item the function  $f : \mathbb{C}^\ast \to  \mathbb{C}^\ast$ defined by $f(\lambda) = g(\lambda)^n \sigma (\lambda)^{-1}$, $\lambda \in \mathbb{C}^\ast$  is an automorphism of the multiplicative group $\mathbb{C}^\ast$ and
\begin{equation}\label{E:glcomplex2}
\phi (A) = g(\det A) T(A_{\sigma}^{-1})^t T^{-1} , \ \ \ A \in  GL_n (\mathbb{C}).
\end{equation}
\end{itemize}
\end{proposition}

We will call the maps of $GL_n(\mathbb F)$ of the forms \eqref{jao}, \eqref{E:glcomplex1} automorphisms of the first type and the maps of the forms \eqref{jaojao}, \eqref{E:glcomplex2} automorphisms of the second type.

To verify the statements above, we recall that by the general result \cite[Theorem 1]{Die}, every automorphism $\phi$ of the group $GL_n (\mathbb{F})$ is either of the form 
$$
\phi (A) = \chi (A) TA_\sigma T^{-1}
$$ 
or of the form 
$$
\phi (A) = \chi (A) T(A_{\sigma}^{-1})^t T^{-1},
$$ 
where
$\chi : GL_n (\mathbb{F}) \to \mathbb{F}^\ast$ is a multiplicative function, $\sigma : \mathbb{F} \to \mathbb{F}$ is a field automorphism, and $T\in GL_n(\mathbb F)$. However, to be precise, we have to make a remark here. Indeed, one may notice that the formulation given in \cite{Die} is not the same as the one above. Namely, in \cite{Die}, linear and semilinear operators (i.e., maps which are additive and homogeneous relative to a field automorphism) are considered while here we treat matrices. Due to the well-known fact that the only automorphism of the field $\mathbb R$ is the identity, in the real case there is absolutely no problem to translate those results concerning linear operators to results concerning matrices. However, in the complex case some explanation is needed.

First, we recall the precise definition of a semilinear map.  A map $T : \mathbb{C}^n \to
\mathbb{C}^n$ is said to be semilinear if there exists an automorphism $\sigma$ of the field $\mathbb C$ such that $T(x+y) = Tx +Ty$ and $T(\lambda x) = \sigma (\lambda) Tx$ for every pair of vectors  $x,y \in \mathbb{C}^n$ and every $\lambda \in \mathbb{C}$. In this case, we say that $T$ is semilinear relative to $\sigma$.
Assume now that $T$ is a bijective semilinear map relative to $\sigma$. Then the transformation $A \mapsto TAT^{-1}$, $A \in  GL_n (\mathbb{C})$ is an automorphism of $GL_n (\mathbb{C})$ (here the variable $A$ is considered as a linear operator). If we denote by $J : \mathbb{C}^n \to
\mathbb{C}^n$ the map defined by $Jx = x_\sigma$, $x \in \mathbb{C}^n$, then $T = T'J$ for some bijective linear operator  $T' : \mathbb{C}^n \to
\mathbb{C}^n$. The transformation  $A \mapsto TAT^{-1}$ can be rewritten as  $A \mapsto T' (JAJ^{-1})T'^{-1}$. It is easy to verify that the matrix that corresponds to the invertible linear operator $JAJ^{-1}$ is $A_\sigma$. Hence, 
the automorphisms of the first type that appear in
Dieudonn\'e's result in \cite{Die}, i.e., the transformations $\phi : GL_n (\mathbb{C}) \to GL_n (\mathbb{C})$ of the form 
$$
\phi (A) = \chi ( A) TA T^{-1} , \ \ \ A \in  GL_n (\mathbb{C}),
$$
where $T : \mathbb{C}^n \to
\mathbb{C}^n$ is a bijective semilinear map 
and $\chi: GL_n (\mathbb{C})\to \mathbb{C}^\ast$ is a multiplicative function, take the matrix form 
$$
\phi (A) = \chi (A) TA_\sigma T^{-1}, \quad A\in GL_n(\mathbb C),
$$ 
(the matrix $T$ that appears here should not be confused with the bijective semilinear map $T$ that shows up above).
With some more efforts, the automorhisms of the second type in \cite[Theorem 1]{Die} (involving the notion of adjoint maps on dual spaces) can be verified to transform 
to the matrix form 
$$
\phi (A) = \chi (A) T(A_{\sigma}^{-1})^t T^{-1}, \quad A\in GL_n(\mathbb C).
$$ 

To arrive at the forms given in Propositions \ref{P:GLaut} and \ref{P:GLautcomplex}, we need to have a closer look also at the multiplicative function $\chi$.
Let us consider first the automorphisms of the form
$
\phi (A) = \chi (A) TA_\sigma T^{-1}, \quad A\in GL_n(\mathbb F).
$
It is well-known that the derived subgroup (i.e., the commutator subgroup) of $ GL_n (\mathbb{F})$ equals the special linear group $SL_n (\mathbb{F})$, that is, the subgroup of $GL_n (\mathbb{F})$ consisting of all matrices with determinant $1$, see, for example, \cite[Theorem II.9.4]{Sup}. It follows that $\chi (A) = 1$ for every $A\in GL_n(\mathbb F)$ with $\det A = 1$. Consequently, 
we have $\chi (A) = \chi (B)$ whenever $\det A = \det B$, $A,B\in GL_n(\mathbb F)$.
Hence, $\chi : GL_n (\mathbb{F}) \to \mathbb{F}^\ast$ is of the form $\chi(A) = g(\det A)$, $A\in GL_n(\mathbb F)$ for some endomomorphism $g$ of the multiplicative group $\mathbb{F}^\ast$. Every automorphism $\phi$ of  $GL_n (\mathbb{F})$ maps the center $\{\lambda I \, | \, \lambda \in \mathbb{F}^\ast \}$ of  $GL_n (\mathbb{F})$ onto itself. It follows that  the multiplicative function $\lambda \mapsto  g(\lambda)^n \sigma(\lambda)$ is bijective. The same argument can be applied to the automorphisms of the form $\phi (A) = \chi (A) T(A_\sigma^{-1})^t T^{-1}$, $A\in GL_n(\mathbb F)$. This proves the necessity parts of our propositions above. The sufficiency parts are easy to check.

In the real case, we can go further in the analysis of the scalar functions $g$ that appear above. We introduce the following classes of functions
\begin{equation}\label{E:functions}
\begin{aligned}
\Mre=&\{g:\mathbb R^\ast \to \mathbb R^\ast\,  |\, g \text{ is multiplicative, and } f(\lambda) = g(\lambda)^n \lambda, \, \lambda \in \mathbb{R}^\ast \\ &\text{ is an automorphism of } \mathbb{R}^\ast\}, \\
\Mrk=&\{g:\mathbb R^\ast \to \mathbb R^\ast\, |\, g \text{ is multiplicative, and } f(\lambda) = g(\lambda)^n \lambda^{-1}, \, \lambda \in \mathbb{R}^\ast \\ &\text{ is an automorphism of } \mathbb{R}^\ast\}.
\end{aligned}
\end{equation}

We have the following statement which is verified by using well-known ideas. 

\begin{observation}\label{polakur}
Let $g : \mathbb{R}^\ast \to  \mathbb{R}^\ast$ be any function. The following assertions are equivalent.
\begin{itemize}
\item $g\in \Mre$. 
\item  There exists a bijective additive function $a : \mathbb{R} \to \mathbb{R}$ such that 
\begin{equation}\label{E:gm1}
g(\mu) = { 1 \over \sqrt[n]{\mu}} e^{a ( \log \mu) }
\end{equation}
for every real $\mu > 0$, and if $n$ is odd then $g( \mu) =  g(-\mu)$ for every real $\mu < 0$, and if $n$ is even then 
either  $g( \mu) = - g(-\mu)$ for every real $\mu < 0$ or  $g( \mu) = g(-\mu)$ for every real $\mu < 0$.
\end{itemize}
\end{observation}

We start the verification of the above by assuming that $g\in \Mre$. Define $f : \mathbb{R}^\ast \to  \mathbb{R}^\ast$ by $f(\lambda) =  g(\lambda)^n \lambda$, $\lambda \in  \mathbb{R}^\ast$. We know that $f$ is an automorphism of $\mathbb R^\ast$. From $f(\mu^2) = (f (\mu))^2$, $\mu
\in \mathbb{R}^\ast$, we conclude that $f$ maps the interval $(0, \infty)$ into itself. The inverse of $f$ is also an automorphism of $\mathbb{R}^\ast$, and, therefore, $f$ restricted to $(0, \infty)$ is an automorphism of the multiplicative group $(0, \infty)$. It is then clear that the function $b : \mathbb{R} \to \mathbb{R}$
defined by
$$
b(\mu) = \log f (e^\mu), \ \ \ \mu \in \mathbb{R},
$$
is a bijective additive function of $\mathbb R$. Thus,
$$
f(\mu) = e^{b (\log \mu)}, \quad \mu>0.
$$
By the multiplicativity of $g$, it also maps $(0,\infty)$ into itself. Then it follows that 
for every real $\mu > 0$ we have
\begin{equation}\label{hhaass}
g(\mu) = { 1 \over \sqrt[n]{\mu}} e^{a ( \log \mu) },
\end{equation}
where $a : \mathbb{R} \to \mathbb{R}$ is a bijective additive function defined by $a = (1/n)b$.

The multiplicativity of $g$ yields that either $g(-1) = 1$, or $g(-1) = -1$. It is now trivial to complete the proof in this direction and the other direction is an easy computation, hence omitted.

Apparently, one can get a description of similar spirit for the elements of $\Mrk$. There the formula given in \eqref{E:gm1} changes to $g(\mu) = 
{\sqrt[n]{\mu}} e^{a ( \log \mu)}$, $\mu>0$.

We can provide even more information on the functions that belong to $\Mre$. 
Let $g : \mathbb{R}^\ast \to  \mathbb{R}^\ast$ be a function satisfying one, and hence both of the conditions in Observation \ref{polakur}, and let us define the function $f : \mathbb{R}^\ast \to  \mathbb{R}^\ast$ in the same way as in the above proof, i.e., $f(\lambda) =  g(\lambda)^n \lambda$, $\lambda \in  \mathbb{R}^\ast$. Then $f$ maps the set of all nonzero real numbers bijectively onto itself and it also maps the real positive half-line bijectively onto itself. We know that $g$ maps the real positive half-line to itself and the negative half-line either to itself or to the positive half-line. But the restriction of $g$ to the positive half-line is not necessarily a bijective map of the positive half-line onto itself. The first obvious example is the constant function $g(\lambda) = 1$, $\lambda \in \mathbb{R}^\ast$. Actually, there are plenty of examples. 
In the above proof we have used the fact that the multiplicative group of all positive real numbers is isomorphic to the additive group of all real numbers, where the isomorphism and its inverse are the logarithmic function and the exponential function. Thus, the question whether $g$ is a bijection of $(0, \infty)$ onto itself boils down (see (\ref{hhaass})) to the question wheter the additive function
$$
\lambda \mapsto a(\lambda) - {1 \over n}\lambda, \ \ \ \lambda \in \mathbb{R}
$$
is bijective. And this is not true in general. Indeed, if we take a Hamel basis $\{ e_\alpha \, | \, \alpha \in J \}$ of the vector space $\mathbb{R}$ over the field of rational numbers and write $J$ as a disjoint union $J = J_1 \cup J_2$ and then define a bijective additive function $a : \mathbb{R} \to \mathbb{R}$
by
$$
a( e_\alpha ) = e_\alpha
$$
whenever $\alpha \in J_1$ and
$$
a( e_\alpha ) =  {1 \over n} e_\alpha
$$
whenever $\alpha \in J_2$ (we know that $a$ is a linear map on $\mathbb{R}$ considered as a vector space over $\mathbb{Q}$, and hence it is well-defined once we know its values on the chosen Hamel basis), then the null space of the $\mathbb{Q}$-linear function $\lambda \mapsto a(\lambda) - {1 \over n}\lambda$, $\lambda \in \mathbb{R}$ is the linear span of $\{ e_\alpha \, | \, \alpha \in J_2 \}$ and its image is the linear span of $\{ e_\alpha \, | \, \alpha \in J_1 \}$. The previous example $g(\lambda) = 1$, $\lambda \in \mathbb{R}^\ast$ is represented by the special case where $J_1 = \emptyset$.
Apparently, similar consideration can be made concerning the elements of $\Mrk$.

So, we have a kind of good picture about the functions $g$ that appear in the description of the automorphisms of $GL_n(\mathbb R)$ given in Proposition \ref{P:GLaut}. As for the complex case, the situation is much different. First, recall that already the structure of the field automorphisms of $\mathbb C$ is quite complicated. While, as mentioned before, $\mathbb R$ has only one field automorphism, the identity, $\mathbb C$ is known to have $2^\mathfrak{c}$ automorphisms (meaning that there are as many field automorphisms of $\mathbb C$ as arbitrary functions from $\mathbb C$ into itself), so we do not expect to have any useful description of the group automorphisms of $\mathbb C^\ast$, not to mention the functions $g$ that appear in Proposition \ref{P:GLautcomplex}.

The above was about the automorphisms of the general linear group.
Let us now consider the case of the special linear group $SL_n (\mathbb{F})$.
Concerning its automorphisms we recall that they are exactly the restrictions of automorphisms of $GL_n (\mathbb{F})$, see \cite[Theorem 5]{Die}. Therefore, by Propositions \ref{P:GLaut} and \ref{P:GLautcomplex}, we have the precise forms of the automorphisms of $SL_n (\mathbb{F})$.

\begin{proposition}\label{P:SL}
The automorphisms of $SL_n (\mathbb{R})$ are exactly the maps of the forms
\begin{equation}\label{F:SLR}
A\mapsto TAT^{-1}, \qquad A\mapsto T(A^{-1})^t T^{-1}
\end{equation}
where $T\in GL_n (\mathbb{R})$. 

The automorphisms of $SL_n (\mathbb{C})$ are exactly the maps of the forms
\begin{equation}\label{F:SLC}
A\mapsto TA_\sigma T^{-1}, \qquad A\mapsto T(A_\sigma^{-1})^t T^{-1}
\end{equation}
where $T\in GL_n (\mathbb{C})$, and $\sigma$ is a field automorphism of $\mathbb{C}$.
\end{proposition}

We next turn to the automorphisms of the unitary group $U_n(\mathbb C)$ of $n\times n$ complex unitary matrices.
We denote by $\Se$ the multiplicative group of all complex numbers of modulus 1, i.e., $\Se = \{ z \in \mathbb{C} \, | \, |z| = 1 \}$. 
The structure of the automorphisms of the group $U_n(\mathbb C)$ can be described as follows.

\begin{proposition}\label{P:autunitary}
The map $\phi$ of the group $U_n (\mathbb{C})$ is an automorphism if and only if there exist $T \in U_n (\mathbb{C})$
and a multiplicative function $g : \Se \to  \Se$ such that the function  $f : \Se \to  \Se$ defined by $f(\lambda) = g(\lambda)^n \lambda $, $\lambda \in \Se$ is an automorphism of the multiplicative group $\Se$ and either we have 
\begin{itemize}
\item
\begin{equation*}\label{F:unitary1}
\phi (A) = g(\det A) TA T^{-1} , \ \ \ A \in  U_n (\mathbb{C}),
\end{equation*}
or we have
\item 
\begin{equation*}\label{F:unitary2}
\phi (A) = g(\det \overline{A}) T\overline{A} T^{-1} , \ \ \ A \in  U_n (\mathbb{C}).
\end{equation*}
\end{itemize}
Here, $\overline{A}$ denotes the matrix obtained from $A$ by applying the complex conjugation entrywise.
\end{proposition}

To verify the "only if" part of the statement, we assume that $\phi$ is an automorphism of $U_n (\mathbb{C})$ and recall that, by the main theorem in \cite{Joh}, there exist an orthogonality preserving bijective semilinear map $T : \mathbb{C}^n \to \mathbb{C}^n$ and a multiplicative map $\chi$ from $U_n (\mathbb{C})$ into $\Se$ such that $\phi (A) = \chi (A) T A T^{-1}$, $A\in U_n (\mathbb{C})$. By the orthogonality preserving property of $T$, it follows from \cite[Lemma 2.1]{RoS} that there exists a constant $c \in \mathbb{C}$ such that either $T$ is linear and $[ Tx, Ty ] = c [ x,y ]$, $x, y \in \mathbb{C}^n$, or $T$ is conjugate-linear and $[ Tx, Ty ] = c [ y,x ]$, $x, y \in \mathbb{C}^n$. Here, $[ \cdot , \cdot]$ denotes the standard inner product on $\mathbb{C}^n$.
Inserting $x=y$, we see that $c > 0$, and after replacing $T$ by $(1 / \sqrt{c})T$, we can assume that $c=1$. Thus, it remains to show that in the linear case we have $\chi (A) = g( \det A)$, $A\in U_n (\mathbb{C})$ for some multiplicative function $g: \Se \to \Se$ such that the function $f: \Se \to \Se$ defined by $f (\lambda ) = g(\lambda)^n \lambda$, $\lambda \in \Se$ is an automorphism of  $\Se$. (The conjugate-linear case can either be treated in the same way or can be reduced to the linear case by composing the original map with the automorphism given by the conjugation $A\mapsto \overline{A}$.) It is enough to verify that $\chi (A) = 1$ for every unitary $A\in U_n(\mathbb C)$ with determinant 1, the rest of the proof goes through as in the case of general linear groups. Let us call any unitary matrix of the form $I - 2P$, where $P$ is a projection (i.e., a Hermitian idempotent) of rank one, a simple symmetry. The simple symmetries are all unitarily similar,  therefore we have $\chi (B) = \chi (C)$ for any pair $B,C$ of simple symmetries. Moreover, we have $B^2 = I$, and therefore, we conclude that either $\chi (B) = 1$ for every simple symmetry $B$, or 
$\chi (B) = -1$ for every simple symmetry $B$. 
By \cite[Theorem 9]{Rad}, every unitary matrix with determinant $\pm 1$ is a product of simple symmetries. Since we are considering $A\in U_n(\mathbb C)$ with $\det A = 1$, the number of factors in this decomposition must be even. It follows that $\chi (A) = 1$, as desired.

The "if" part of the statement in the proposition is easy, we omit it.

As for the functions $g$ appearing in Proposition \ref{P:autunitary}, we tell that, like in the case of $GL_n(\mathbb C)$, we do not have any useful, more detailed description concerning them. We only mention the fact that the group $\Se$ is actually isomorphic to $\mathbb C^\ast$ which was proved in \cite{Clay}.

We will also need the 
structure of the automorphisms of the special unitary group $SU_n(\mathbb C)$.
That group has essentially only one outer automorphism (see, e.g., \cite{Baum} p. 201),
therefore, we have the following description.

\begin{proposition}\label{P:SU}
The automorphisms of $SU_n(\mathbb C)$ are exactly the transformations 
\begin{equation}\label{F:su}
A\mapsto TAT^{-1}, \qquad A\mapsto T\overline{A}T^{-1},
\end{equation}
where $T\in U_n(\mathbb C)$.
\end{proposition}



\section{The local automorphisms of the considered classical groups}

This section contains the main results of the paper. The first two  concern the local automorphisms of the special linear groups $SL_n (\mathbb{R})$, $SL_n (\mathbb{C})$.
Our result on the local automorphisms of $SL_n (\mathbb{R})$ can be formulated very simply.

\begin{theorem}\label{specialreal}
Every local automorphism $\phi : SL_n (\mathbb{R}) \to SL_n (\mathbb{R})$ is an automorphism.
\end{theorem}

Before presenting the proof, for the sake of comparison, let us immediately formulate the corresponding result for $SL_n (\mathbb{C})$. 

\begin{theorem}\label{specialcomplex}
For any map $\phi : SL_n (\mathbb{C}) \to SL_n (\mathbb{C})$ the following assertions are equivalent.
\begin{itemize}
\item $\phi$ is a local automorphism. 
\item  There exist $T \in GL_n (\mathbb{C})$ and a field endomorphism $\sigma$ of $\mathbb C$ such that either
$$
\phi (A) =  TA_\sigma T^{-1} , \ \ \ A \in  SL_n (\mathbb{C}),
$$
or
$$
\phi (A) =  T(A_{\sigma}^{-1})^t T^{-1} , \ \ \ A \in  SL_n (\mathbb{C}).
$$
\end{itemize}
\end{theorem} 

Observe that above $\sigma$ is not necessarily a field automorphism (cf. the second part of Proposition \ref{P:SL}) only a field endomorphism.

In what follows we present a common proof of both of the previous two results, and later we will give a short argument for the real case. 
In the common proof we will make use of the next four auxiliary statements. 

\begin{lemma}\label{functional}
Let $\sigma$ be an endomorphism of the field $\mathbb{F}$, $\varphi_1 : \mathbb{F}^n \to \mathbb{F}$ a linear functional, and $\varphi_2 : \mathbb{F}^n \to \mathbb{F}$ a semilinear functional relative to $\sigma$. (This latter means that $\varphi_2$ is additive and satisfies $\varphi_2(\lambda x)=\sigma(\lambda)\varphi_2(x)$ for all $x\in \mathbb F^n$ and $\lambda \in \mathbb F$.) Suppose that 
$$
{\rm Ker}\, \varphi_1 = {\rm Ker}\, \varphi_2 .
$$
Then there exists a nonzero constant $c \in \mathbb{F}$ such that $\varphi_2 (x) =  c \sigma (\varphi_1 (x))$ for every $x\in \mathbb{F}^n$. 
\end{lemma}

\begin{proof} 
There is nothing to prove when ${\rm Ker}\, \varphi_1 = {\rm Ker}\, \varphi_2  = \mathbb{F}^n$. So, assume that the functionals $\varphi_1,\varphi_2$ are nonzero. Let $u\not\in {\rm Ker}\, \varphi_1$ be a vector such that $\varphi_1 (u) = 1$ and set $c = \varphi_2 (u)$. Then for every $x \in \mathbb{F}^n$ there exist
a unique $\lambda \in \mathbb{F}$ and a unique $v \in {\rm Ker}\, \varphi_1$ such that $x = \lambda u + v$. We have
$$
\varphi_2 (x) = \varphi_2 (\lambda u + v) = \sigma (\lambda) \varphi_2 (u) = c \sigma (\lambda) = c \sigma (\varphi_1 (x)),
$$
as desired.
\end{proof}

\begin{lemma}\label{L:lindep}
Let $\sigma$ be a nonzero endomorphism of the field $\mathbb F$. Assume that $A,B\in GL_n(\mathbb F)$ are such that $Ax_\sigma, Bx_\sigma$ are linearly dependent for all $x\in \mathbb F^n$. (Recall that $x_\sigma=(x_i)_\sigma=(\sigma(x_i))$.) Then $A,B$ are linearly dependent.
\end{lemma}

\begin{proof}
For a similar result, see \cite[Theorem 2.3]{BrS}, but observe that our present lemma does not follow from that result directly since $\sigma$ is not assumed to be surjective.

Observe that for any nonzero $x\in \mathbb{F}^n$, we have that $x_\sigma$ is nonzero (any nonzero field endomorphism is injective).
By our assumption, for every nonzero $x\in \mathbb{F}^n$, we have a nonzero scalar $\lambda_x\in \mathbb F$ such that $A x_\sigma = \lambda_x B x_\sigma$. Take linearly independent vectors $x, u \in \mathbb{F}^n$. Then
$$
\lambda_{x+u} B x_\sigma + \lambda_{x+u} B u_\sigma = \lambda_{x+u} B (x + u)_\sigma = A (x+u)_\sigma $$ $$= A x_\sigma + A u_\sigma =  \lambda_x B x_\sigma +  \lambda_u B u_\sigma.
$$
Arguing by contraposition, one can check rather easily (but not completely trivially) that because $x$ and $u$ are linearly independent, the vectors $x_\sigma$ and $u_\sigma$ are also linearly independent. Consequently, $B x_\sigma$ and $B u_\sigma$ are linearly independent and it follows that $\lambda_x = \lambda_u$. In the case where $x$ and $u$ are linearly dependent, we can find $w\in \mathbb F^n$ that is linearly independent of $x$ and $u$, and then we have $\lambda_x = \lambda_w =\lambda_u$. Thus, $\lambda = \lambda_x$ is independent of $x$, and we have
$$
A x_\sigma = \lambda B x_\sigma, \ \ \ x \in \mathbb{F}^n .
$$
Putting the elements of the natural basis of $\mathbb F^n$ into the place of $x$, we conclude that $A=\lambda B$.
\end{proof}

Another observation that we will need is the following.

\begin{lemma}\label{L:GNR}
If $A\in M_n(\mathbb F)$ is not a scalar multiple of the identity, then 
$$
\{[Ax,y] \, | \, [ x,y] =1\}=\mathbb F.
$$
\end{lemma}

\begin{proof} Indeed, if $A$ is not a scalar multiple of the identity, then there is a unit vector $x\in \mathbb F^n$ which is not an eigenvector of $A$ (cf. Lemma \ref{L:lindep}). Consequently, the  vector $Ax\in \mathbb F^n$ has a nonzero component $z$ in $[x]^\perp$, the orthogonal complement of the subspace $[x]$ spanned by $x$. Pick an arbitrary $v\in [x]^\perp$. Then $[ x,x+v] =1$ and we compute
\[
[Ax, x+v]=[ Ax,x] +[ z,v].
\]
Clearly, the quantity on the right hand side can be any element of the scalar field $\mathbb F$.
\end{proof}

Finally, we will need the following statement about the extendability of nonzero homomorphisms of finitely generated subfields of $\mathbb C$ to field automorphisms of $\mathbb C$.
For any nonempty set $S \subset \mathbb{C}$, we denote by $\mathbb{Q} (S)$ the subfield of $\mathbb{C}$ generated by $S$.

\begin{proposition}\label{endom}
Let $S \subset \mathbb{C}$ be a finite set and $\sigma : \mathbb{Q}(S) \to \mathbb{C}$ a nonzero field homomorphism. Then there exists a field automorphism $\tau : \mathbb{C} \to \mathbb{C}$ such that $\sigma (\lambda) = \tau (\lambda)$ for every $\lambda \in \mathbb{Q}(S)$.
\end{proposition}

\begin{proof}
All the ideas needed to verify this statement can be found in $\cite{Kes}$. We will just sketch the proof. Let $T_1$ be a transcendence basis of $\mathbb{C}$ considered as a field extension of $\mathbb{Q}(S)$. Similarly, let  $T_2$ be a transcendence basis of $\mathbb{C}$ considered as a field extension of $\sigma (\mathbb{Q}(S))$. Because $S$ is finite, both field extensions have transcendence degree continuum and therefore there exists a bijection $\xi : T_1 \to T_2$. In a natural (and unique) way, one can define a field isomorphism from $\mathbb{L} = (\mathbb{Q}(S)) \, (T_1)$ onto $\mathbb{K} =( \sigma (\mathbb{Q}(S))) (T_2)$ (for the details see \cite{Kes}) whose restriction to $\mathbb{Q}(S)$ is $\sigma$ and whose restriction to $T_1$ acts as $\xi$. We use the same symbol $\sigma$ to denote this extension.

It was proved in \cite{Kes} that if $\mathbb{M}$ is a subfield of $\mathbb{C}$, $\rho : \mathbb{M} \to \mathbb{C}$ is a field homomorphism, and $\lambda \in \mathbb{C}$ is algebraic over $\mathbb{M}$, then there exists a field homomorphism $\overline{\rho} : \mathbb{M} ( \lambda) \to \mathbb{C}$ that extends $\rho$. A straightforward application of Zorn's lemma yields the existence of a field endomorphism $\tau : \mathbb{C} \to \mathbb{C}$ such that $\sigma (\lambda) = \tau (\lambda)$ for every $\lambda \in \mathbb{L}$. In particular, $\sigma (\lambda) = \tau (\lambda)$ for every $\lambda \in \mathbb{Q}(S)$ and $\mathbb{K} \subset \tau (\mathbb{C})$. We know that $\tau (\mathbb{C})$ is algebraically closed and $\mathbb{C}$ is algebraic over $\tau (\mathbb{C})$. It follows that $\tau (\mathbb{C}) = \mathbb{C}$, as desired. 
\end{proof}

After these preparations, we are now in a position to give the common proof of our results on the local automorphisms of the special linear groups both in the real and the complex cases.

\begin{proof}[Proof of Theorems \ref{specialreal} and \ref{specialcomplex}]
In what follows, $\Pne$ stands for the set of all $n \times n$ idempotent matrices in $M_n(\mathbb F)$ of rank one.

Let us denote by ${\cal E} \subset  SL_n (\mathbb{F})$ the set of all matrices of the form $(1/2)^{n-1}P + 2(I-P)$, where $P\in \Pne$. Any such $P$ can be written as $P = xy^t$, where
$x,y \in \mathbb{F}^n$ with $y^t x=1$. If $\theta$ is a field automorphism of $\mathbb{F}$, then the map $A \mapsto A_\theta$ from $M_n(\mathbb F)$ into itself is bijective, additive, and multiplicative.
Note that in the real case $\theta$ is the identity function, while in the complex case $\theta$ fixes all rationals. In particular, the matrix
$$
[ (1/2)^{n-1}xy^t + 2(I-xy^t) ]_\theta = 
(1/2)^{n-1}x_\theta y_{\theta}^t + 2(I- x_\theta y_{\theta}^t )
$$ 
has eigenvalues $(1/2)^{n-1}$ and $2$, and
$$
\left([ (1/2)^{n-1}xy^t + 2(I-xy^t) ]_\theta^{-1} \right)^{t} = 
2^{n-1}y_\theta x_{\theta}^t + (1/2) (I- y_\theta x_{\theta}^t)
$$ 
has eigenvalues $2^{n-1}$ and $1/2$.

Let $A,B \in {\cal E}$. By the structure of the automorphisms of $SL_n(\mathbb F)$, see \eqref{F:SLR}, \eqref{F:SLC}, there exist $S \in GL_n (\mathbb{F})$ and a field automorphism $\theta$ of $\mathbb{F}$ such that either
$$
\phi (A) =  SA_\theta S^{-1}  \ \ \  {\rm and} \ \ \  \phi (B) =  SB_\theta S^{-1} ,
$$
or
$$
\phi (A) =  S(A_{\theta}^{-1})^t S^{-1}  \ \ \  {\rm and} \ \ \  \phi (B) =  S(B_{\theta}^{-1})^t S^{-1}.
$$
Hence, either both $\phi (A)$ and $\phi (B)$ have eigenvalues $(1/2)^{n-1}$ and $2$, or 
both $\phi (A)$ and $\phi (B)$ have eigenvalues $2^{n-1}$ and $1/2$. It follows that either for every $A \in {\cal E}$ the matrix 
$\phi (A)$ has eigenvalues $(1/2)^{n-1}$ and $2$, or for every $A \in {\cal E}$ the matrix 
$\phi (A)$ has eigenvalues $2^{n-1}$ and $1/2$. After replacing our local automorphism $\phi$ by the local automorphism $A \mapsto (\phi(A)^{-1})^t$ if necessary, we may assume with no loss of generality that we have the first case.

In particular, $\phi ({\cal E}) \subset {\cal E}$. That is, if $A = (1/2)^{n-1}P + 2(I-P) \in {\cal E}$, then  $\phi (A) = (1/2)^{n-1}Q + 2(I-Q) \in {\cal E}$. Of course, the rank-one idempotent $Q\in \Pne$ is uniquely determined by the rank-one idempotent $P\in \Pne$. Hence, our local automorphism $\phi$ induces a map $\xi : \Pne \to \Pne$ such that 
$$
\phi ( (1/2)^{n-1}P + 2(I-P) ) = (1/2)^{n-1}\xi (P) + 2(I-\xi (P)), \quad P\in \Pne.
$$ 
We introduce some more notation. We identify $\mathbb{F}^n$ with the set of all $n \times 1$ matrices. As before, for a nonzero $x \in \mathbb{F}^n$ we denote by $[x]$ the one-dimensional subspace of $\mathbb{F}^n$ spanned by $x$. The projective space over $\mathbb{F}^n$ will be denoted by $\mathbb{P} (\mathbb{F}^n)$, 
$$
\mathbb{P} (\mathbb{F}^n) = \{ [x] \, | \, x \in \mathbb{F}^n \setminus \{0 \} \}.
$$
Similarly, the projective space over the space of all row matrices of the size $1 \times n$ is defined by
$$
\mathbb{P} ((\mathbb{F}^n)^t) = \{ [y^t] \, | \, y \in \mathbb{F}^n \setminus \{0 \} \}.
$$

We next show that the map $\xi$ gives rise to two maps, one on $\mathbb{P} (\mathbb{F}^n)$ and one on $\mathbb{P} ((\mathbb{F}^n)^t)$.
If we consider $P = xy^t \in \Pne$ as a linear operator, then its image ${\rm Im}\, P$ is $[x]$ and its null space ${\rm Ker}\, P$ is $[y^t]^\perp = \{  u \in \mathbb{F}^n \, | \, y^t u = 0 \}$. Let $P,Q \in \Pne$. 
Since $\phi$ is a local automorphism, there exist $S \in GL_n (\mathbb{F})$ and a field automorphism $\theta$ of $\mathbb{F}$ such that 
$$
\phi ((1/2)^{n-1}P + 2(I-P))  = (1/2)^{n-1}\xi (P) + 2(I-\xi (P))  $$ $$=  S ((1/2)^{n-1}P_\theta + 2(I-P_\theta)) S^{-1}  
$$
and
$$
\phi ((1/2)^{n-1}Q + 2(I-Q  ))=
(1/2)^{n-1}\xi (Q) + 2(I-\xi (Q)) $$ $$ =  S((1/2)^{n-1}Q_\theta + 2(I-Q_\theta)) S^{-1}  .
$$ 
It follows that if $P$ and $Q$ have the same image, that is, $P= xy^t$ and $Q= (\lambda x) u^t$ for some vectors $x,y, u \in  \mathbb{F}^n$ and nonzero scalar $\lambda\in \mathbb F$ with $y^t x = u^t (\lambda x) = 1$, then 
$$
\xi (P) = S P_\theta S^{-1} = (Sx_\theta) (y_{\theta}^t S^{-1}) \quad \text{ and } \quad \xi (Q) = \theta (\lambda) (Sx_\theta) (u_{\theta}^t S^{-1}), 
$$  
so $\xi (P), \xi (Q)$ have the same 
image. Similarly, if $P$ and $Q$ have the same null space, then this is true also for $\xi (P)$ and $\xi (Q)$. 

It follows that the map $\xi :\Pne \to \Pne$ induces two maps
$$
\tau : \mathbb{P} (\mathbb{F}^n) \to \mathbb{P} (\mathbb{F}^n)
\ \ \
{\rm and}
\ \ \
\eta : \mathbb{P} ((\mathbb{F}^n)^t) \to \mathbb{P} ((\mathbb{F}^n)^t)
$$
such that for every $P = xy^t \in \Pne$ we have
$$
{\rm Im}\, \xi (P) = \tau ([x]) \ \ \ {\rm and} \ \ \ {\rm Ker}\, \xi (P) = (\eta ([y^t]))^\perp .
$$

Our next observation is that if for a pair of rank-one idempotents $P,Q \in \Pne$ we have $PQ = QP = 0$, then $\xi (P) \xi (Q) = \xi(Q) \xi (P) = 0$. To see this, first observe that since $\phi$ is a local automorphism, for every pair $A,B \in SL_n (\mathbb{F})$ we have
$$
AB = BA \iff \phi (A) \phi (B) = \phi (B) \phi (A).
$$
Now, if $PQ = QP = 0$ for some $P,Q \in \Pne$ and if $A= (1/2)^{n-1}P + 2(I-P)$ and $B= (1/2)^{n-1}Q + 2(I-Q)$, then $\phi(A)$ and $\phi (B)$ commute,  and therefore, $\xi (P)$ and $\xi (Q)$ commute. It is an elementary linear algebra exercise to show that if two rank-one idempotents commute, then either they are identical or both of their products are zero. If $\xi (P) = \xi (Q)$, then $\phi ((1/2)^{n-1}P + 2(I-P)) = \phi ((1/2)^{n-1}Q + 2(I-Q))$, and therefore, $P=Q$, a contradiction. Hence, $\xi$ has the desired property.

We will next show that for every nonzero vectors $x,u,v \in \mathbb{F}^n$, we have
$$
[x] \subset [u] + [v] \Rightarrow \tau ([x]) \subset \tau ([u]) + \tau ([v]).
$$
There is nothing to prove if $u$ and $v$ are linearly dependent. So, assume they are linearly independent. Then we can find rank-one idempotents $P_1 = uy_{1}^t$, $P_2 = vy_{2}^t$, $P_3 = x_3 y_{3}^t, \ldots, P_n = x_n y_{n}^t$ such that $P_i P_j = 0$ whenever $i \not= j$. Moreover, we can find a rank-one idempotent $Q = x y^t$ such that $QP_j = P_j Q = 0$ for $j=3, \ldots , n$. Denote $\xi (P_j) = R_j$, $j=1, \ldots , n$, and $\xi (Q) = Q'$. Then $R_i R_j = 0$ whenever $i \not= j$ and $Q' R_j = R_j Q' = 0$ for $j=3, \ldots , n$. In particular,
$$
 \tau ([u]) + \tau ([v]) = {\rm Im}\, R_1 \oplus  {\rm Im}\, R_2 = {\rm Ker}\, R_3 \cap \ldots \cap {\rm Ker}\, R_n
$$
and since $R_j Q' = 0$ for $j=3, \ldots , n$, we have $\tau ([x]) = {\rm Im}\, Q' \subset {\rm Ker}\, R_3 \cap \ldots \cap {\rm Ker}\, R_n$. Hence, we have the desired inclusion $ \tau ([x]) \subset \tau ([u]) + \tau ([v])$.

We can apply the nonsurjective version of the fundamental theorem of projective geometry \cite[Theorem 3.1]{Fau} to conclude that there exists a nonzero endomorphism $\sigma$ of the field $\mathbb{F}$ and a matrix $T\in M_n(\mathbb F)$ such that
$$
\tau ([x]) = [ T x_\sigma]
$$
for every nonzero $x \in \mathbb{F}^n$. (We should point out that in the paper \cite{Fau}, a semilinear operator with respect to the endomorphism $\sigma$ appears, while above we have a corresponding matrix $T$. In fact, the $j$th column of $T$ consists of the coordinates of the image of the $j$th element of the standard basis of $\mathbb F^n$ under that semilinear operator.)  The matrix $T$ is invertible because the image of $T$ can be shown to contain $n$ linearly independent vectors. To see this, observe that, as we have seen above, the range of $\xi$ contains $n$ rank-one idempotents (think of $R_j$, $j=1,...,n$) with zero products which implies that the sum of those idempotents is the identity. Consequently, selecting any nonzero vectors in the images of those idempotents, we obtain a linearly independent set of $n$ vectors.

In the same way, we see that $\eta ([y^t]) = [y_{\kappa}^t W]$, $y \in \mathbb{F}^n \setminus \{0\}$, for some nonzero field endomorphism $\kappa$ of $\mathbb{F}$ and some invertible matrix $W\in M_n(\mathbb F)$. We next study the connection between $T$ and $W$, and $\sigma$ and $\kappa$.

We claim that for every pair of nonzero vectors $x \in \mathbb{F}^n$ and $y^t \in ( \mathbb{F}^n)^t$ we have
\begin{equation}\label{motorce}
y^t \, x = 0 \iff y_{\kappa}^t WT x_\sigma = 0.
\end{equation}
Indeed, if $y^t \, x \not= 0$, then there exists a unique rank-one idempotent $P\in \Pne$ whose image is $[x]$ and whose null space is $[y^t]^\perp$. It follows that the image of $\xi(P)$ is  $ [ T x_\sigma]$ and the null space of  $\xi (P)$ is $[y_{\kappa}^t W]^\perp$. The image of a rank-one idempotent is not contained in its null space, and therefore, $ y_{\kappa}^t WT x_\sigma \not= 0$.

If, on the other hand,  $y^t \, x = 0$, then we can easily find rank-one idempotents $P,Q\in \Pne$ such that $PQ = QP = 0$ and the image of $Q$ is $[x]$ while the null space of $P$ is $[y^t]^\perp$. From $\xi (P) \xi (Q) =\xi (Q) \xi (P)= 0$ we infer $ y_{\kappa}^t WT x_\sigma = 0$, as desired.

It follows from (\ref{motorce}) and Lemma \ref{functional} that for every nonzero $y \in \mathbb{F}^n$ there exists a nonzero number $c(y) \in \mathbb{F}$ such that
$$
 y_{\kappa}^t WT x_\sigma = c(y) \sigma ( y^t \, x)
$$
for every $x \in \mathbb{F}^n$. We denote by $e_1 , \ldots , e_n$ the standard basis vectors in $\mathbb{F}^n$. Replacing $y$ by $e_i$ and $x$ by $e_j$, $i\not=j$, $1 \le i,j \le n$ above, we see that $WT$ is a diagonal matrix. Denote its diagonal entries by $d_1 , \ldots, d_n$. Choosing $y= \lambda e_1 + e_2$, $x = \mu e_1 + (1 - \lambda \mu ) e_2$, where $\lambda ,\mu\in \mathbb F$ are any scalars, we arrive at
$$
(d_1 \kappa (\lambda) - d_2 \sigma (\lambda) ) \, \sigma (\mu) + d_2 = c (\lambda e_1 + e_2).
$$
It follows that
$$
d_1 \kappa (\lambda) = d_2 \sigma (\lambda)
$$
for every $\lambda\in \mathbb F$. Hence, $d_1 = d_2$ and $\kappa = \sigma$. Similarly, we see that $d_1 = \ldots = d_n = d$. We can replace $T$ by $(1/d)T$. Hence, $WT$ is the identity, or equivalently, $W$ is the inverse of $T$.

Consequently, for any rank-one idempotent $P = xy^t\in \Pne$, we have 
$$
\xi (P) = \xi (xy^t) = T x_\sigma y_{\sigma}^t T^{-1} = TP_\sigma T^{-1},
$$
and therefore,
$$
\phi (A) = TA_\sigma T^{-1}
$$
for every $A\in {\cal E}$. After replacing $\phi$ by $A\mapsto T^{-1} \phi (A) T$ we may assume with no loss of generality that
$$
\phi (A) = A_\sigma, \quad A\in {\cal E}.
$$
In the rest of the proof we will prove that this latter equality holds on the full set $SL_n(\mathbb F)$. We will do it in several steps.

Let $A$ be an arbitrary element of ${\cal E}$ and $B$ any member of $SL_n(\mathbb F)$. We know that there exist $S \in GL_n (\mathbb{F})$ and a field automorphism $\theta$ of $\mathbb{F}$ such that either
$$
\phi (A) =  SA_\theta S^{-1}  \ \ \  {\rm and} \ \ \  \phi (B) =  SB_\theta S^{-1} ,
$$
or
$$
\phi (A) =  S(A_{\theta}^{-1})^t S^{-1}  \ \ \  {\rm and} \ \ \  \phi (B) =  S(B_{\theta}^{-1})^t S^{-1}.
$$
Since we know that $\phi (A) = A_\sigma$ has eigenvalues $(1/2)^{n-1}$ and $2$, the second possibility cannot occur. 

We claim that if $B$ is any diagonalizable member of $SL_n(\mathbb F)$ and all its eigenvalues are rational numbers, then
\begin{equation}\label{matemate2}
\phi (B) = B_\sigma .
\end{equation}
Let $B$ be such a matrix, that is, $B = MDM^{-1}$, where $M$ is invertible and $D$ is a diagonal matrix with the diagonal entries $s_1 , \ldots , s_n \in \mathbb{Q}$ whose product is 1. Set $J = {\rm Diag}\, ( (1/2)^{n-1} , 2, \ldots, 2)$ and $A = MJM^{-1}$. Then, by the local property of $\phi$ and using what we have learned above, there exist $S \in GL_n (\mathbb{F})$ and an automorphism $\theta$ of $\mathbb{F}$ such that 
$$
M_\sigma J M_{\sigma}^{-1} = A_\sigma = \phi (A) = SA_\theta S^{-1} = SM_\theta J M_{\theta}^{-1} S^{-1}
$$
and
\begin{equation}\label{hgfds}
\phi (B) = S B_\theta S^{-1} = SM_\theta D M_{\theta}^{-1} S^{-1}.
\end{equation}
The first of the above two equalities implies that $M_{\sigma}^{-1} S M_\theta$ commutes with $J$ and therefore 
$M_{\sigma}^{-1} S M_\theta$ is of the form
$$
M_{\sigma}^{-1} S M_\theta
 = \left[  \begin{matrix}  * & 0 & 0 & \ldots & 0 \cr    0 & * & * & \ldots & * \cr   0 & * & * & \ldots & * \cr  \vdots & \vdots & \vdots & \ddots & \vdots \cr   0 & * & * & \ldots & * \cr    \end{matrix}                        \right].
$$
It follows from (\ref{hgfds}) that
$$
M_{\sigma}^{-1} \phi (B) M_\sigma = (M_{\sigma}^{-1} S M_\theta) \, D (M_{\sigma}^{-1} S M_\theta)^{-1}
=   \left[  \begin{matrix}  s_1 & 0 & 0 & \ldots & 0 \cr    0 & * & * & \ldots & * \cr   0 & * & * & \ldots & * \cr  \vdots & \vdots & \vdots & \ddots & \vdots \cr   0 & * & * & \ldots & * \cr    \end{matrix}                        \right].
$$
In the same way (putting $(1/2)^{n-1}$ in the second diagonal position of $J$), we see that
$$
M_{\sigma}^{-1} \phi (B) M_\sigma
=   \left[  \begin{matrix}  * & 0 & * & \ldots & * \cr    0 & s_2 & 0 & \ldots & 0 \cr   * & 0 & * & \ldots & * \cr  \vdots & \vdots & \vdots & \ddots & \vdots \cr   * & 0 & * & \ldots & * \cr    \end{matrix}                        \right]
$$
and we can go further to finally conclude that $M_{\sigma}^{-1} \phi (B) M_\sigma$ is equal to $D$, or equivalently, that (\ref{matemate2}) holds.

We now choose and fix an arbitrary matrix $C \in SL_n (\mathbb{F})$ and aim to show that $\phi(C)$ is a scalar multiple of $C_\sigma$. 
To verify this, we first prove that for every $x \in \mathbb{F}^n$, the vectors
$$
C_\sigma x_\sigma  \ \ \ {\rm and} \ \ \ \phi (C) x_\sigma
$$
are linearly dependent. 
Select an arbitrary nonzero $x \in \mathbb{F}^n$ and denote $y= Cx$. We distinguish two cases, the one where $x$ and $y$ are linearly dependent, and the one where they are linearly independent. Let us begin with the first case, that is, assume that $x,Cx$ are linearly dependent, $Cx=\mu x$ holds for some nonzero scalar $\mu \in \mathbb F$.
Then we can find a matrix $L\in GL_n(\mathbb F)$ such that $Lx = e_1$. We next choose a diagonal matrix $B\in SL_n(\mathbb F)$ with pairwise distinct diagonal entries that belong to $\mathbb{Q}$ such that $B$ is not similar to $B^{-1}$. We define a new map $\psi : SL_n (\mathbb{F}) \to SL_n (\mathbb{F})$ by
$$
\psi (X) = L_\sigma \phi (L^{-1}XL) L_{\sigma}^{-1}, \ \ \ X \in SL_n (\mathbb{F}).
$$
Since the product of local automorphisms is a local automorphism, we see that $\psi$ is a local automorphism. Moreover,
$$
\psi (B) =  L_\sigma \phi (L^{-1}BL) L_{\sigma}^{-1} = L_\sigma (L^{-1}BL)_\sigma L_{\sigma}^{-1}= L_\sigma L_{\sigma}^{-1} BL_\sigma L_{\sigma}^{-1} = B.
$$
Applying the fact that $\psi$ is a local automorphism and that $B$ is not similar to $B^{-1}$, we find  
$S \in GL_n (\mathbb{F})$ and an automorphism $\theta$ of $\mathbb{F}$ such that 
$$
\psi (B)  =  SB_\theta S^{-1}=  SB S^{-1}  \ \ \  {\rm and} \ \ \  \psi (LCL^{-1}) =  S(LCL^{-1})_\theta S^{-1}.
$$
We have $B = SB S^{-1}$. Thus, $S$ commutes with $B$ and therefore $S$ is a diagonal matrix, that is, we have $S e_j = \lambda_j e_j$, $j=1, \ldots , n$, for some scalars $\lambda_1 , \ldots , \lambda_n\in \mathbb F$.
Since $Cx=\mu x$, it follows that $LCL^{-1}e_1=\mu e_1$. Applying $\theta$ on the left and the right side of this equality, we arrive at
$$
(LCL^{-1})_\theta e_1 = \theta(\mu)e_1 .
$$
It follows that
$$
 \psi (LCL^{-1})e_1 = S(LCL^{-1})_\theta S^{-1} e_1 = \lambda_{1}^{-1} S(LCL^{-1})_\theta e_1 =
 \lambda_{1}^{-1} S\theta(\mu)e_1 =  \theta(\mu)e_1,
$$
which yields that $L_\sigma \phi (C) L_{\sigma}^{-1} e_1$ and $e_1$ are linearly dependent. Using $x_\sigma =
 L_{\sigma}^{-1} e_1$, we infer that 
$$
\phi (C) x_\sigma \ \ \ {\rm and} \ \ \ x_\sigma
$$
are linearly dependent. Since $Cx=\mu x$, we have that $x_\sigma$ and $C_\sigma x_\sigma$ are linearly dependent, so it finally follows that $\phi(C)x_\sigma$ and $C_\sigma x_\sigma$ are linearly dependent.

Consider now the case where $x$ and $y=Cx$ are linearly independent. We can easily modify our previous argument to treat that case.
First, we can find a matrix $L\in GL_n(\mathbb F)$ such that $Lx = e_1$ and $Ly= e_2$. We next choose a diagonal matrix $B\in SL_n(\mathbb F)$ as above, i.e., with pairwise distinct diagonal entries that belong to $\mathbb{Q}$ such that $B$ is not similar to $B^{-1}$. We again consider the map $\psi : SL_n (\mathbb{F}) \to SL_n (\mathbb{F})$ defined by
$$
\psi (X) = L_\sigma \phi (L^{-1}XL) L_{\sigma}^{-1}, \ \ \ X \in SL_n (\mathbb{F})
$$
and then find  
$S \in GL_n (\mathbb{F})$ and an automorphism $\theta$ of $\mathbb{F}$ such that 
$$
B=\psi (B)  =  SB_\theta S^{-1}=  SB S^{-1}  \ \ \  {\rm and} \ \ \  \psi (LCL^{-1}) =  S(LCL^{-1})_\theta S^{-1}.
$$
As in the former case, it follows that $S$ commutes with $B$ and hence it is a diagonal matrix, that is, we have $S e_j = \lambda_j e_j$, $j=1, \ldots , n$, for some scalars $\lambda_1 , \ldots , \lambda_n \in \mathbb F$.
Next,
$$
(LCL^{-1}) e_1 = LCx = Ly = e_2
$$
and applying $\theta$ on the left and the right side of this equality, we obtain
$$
(LCL^{-1})_\theta e_1 = e_2 .
$$
It follows that
$$
 \psi (LCL^{-1})e_1 = S(LCL^{-1})_\theta S^{-1} e_1 = \lambda_{1}^{-1} S(LCL^{-1})_\theta e_1 =
 \lambda_{1}^{-1} Se_2 =  \lambda_{1}^{-1} \lambda_2 e_2,
$$
which yields that $L_\sigma \phi (C) L_{\sigma}^{-1} e_1=\psi (LCL^{-1})e_1$ and $e_2$ are linearly dependent. Using $x_\sigma =
 L_{\sigma}^{-1} e_1$ and $y_\sigma =  L_{\sigma}^{-1} e_2$, we infer that 
$$
\phi (C) x_\sigma \ \ \ {\rm and} \ \ \ y_\sigma
$$
are linearly dependent, or equivalently, $\phi (C) x_\sigma$ and $C_\sigma x_\sigma$ are linearly dependent.

After this, we apply Lemma \ref{L:lindep} and obtain our claim that $\phi(C)$ is a scalar multiple of $C_\sigma$ for all $C\in SL_n(\mathbb F)$.

Let now $C\in SL_n(\mathbb{F})$ be any matrix that is not a scalar multiple of the identity and let $\phi(C)=\lambda C_\sigma$ for some scalar $\lambda \in \mathbb F$. Consider any element $A$ of $\mathcal E$, that is, select any rank-one idempotent $P\in \Pne$ and set $A=(1/2)^{n-1}P+2(I-P)$. Then we have an invertible matrix $S\in GL_n(\mathbb F)$ and a field automorphism $\theta$ of $\mathbb F$ (they are both depending on $P$) such that
$$
A_\sigma=\phi(A)=SA_\theta S^{-1}, \quad  \lambda C_\sigma=\phi(C)=SC_\theta S^{-1}.
$$
Multiplying those two equalities and then taking trace, we obtain
\begin{equation}\label{E:trace1}
\lambda \sigma(\tr (AC))=\theta(\tr (AC)).
\end{equation}
But 
$$
\tr (AC)=((1/2)^{n-1}-2)\tr(PC)+2\tr(C).
$$
The trace of $C$ is a given scalar, while,
from Lemma \ref{L:GNR}, we learn that $\tr(PC)=\tr(CP)$ can take any value in $\mathbb F$. So, we can find $P\in \Pne$ such that $\tr(AC)=1$ and then \eqref{E:trace1} implies that $\lambda=1$, i.e., we have $\phi(C)=C_\sigma$.

Finally, assume that $C\in SL_n(\mathbb F)$ is a scalar multiple of the identity, i.e., we have $C=\epsilon I$, and $\phi(C)=\lambda C_\sigma$ for some nonzero scalars $\epsilon,\lambda \in \mathbb F$. Choose any $A\in \mathcal E$. Then $\epsilon A$ is not a multiple of the identity and, moreover, it can be checked that $(\epsilon A)_\sigma$ is not similar to $((\epsilon A)_\theta^{-1})^t$ for any automorphism $\theta$ of $\mathbb F$. Then there is $S\in GL_n(\mathbb F)$ and a field automorphism $\theta$ of $\mathbb F$ such that
$$
\sigma(\epsilon)A=(\epsilon A)_\sigma=\phi(\epsilon A)=S(\epsilon A)_\theta S^{-1}=\theta(\epsilon)SA S^{-1}
$$
and
$$
\lambda \sigma(\epsilon)I= \lambda (\epsilon I)_\sigma=\phi(C)=S(\epsilon I)_\theta S^{-1}=\theta(\epsilon) I.
$$
Taking trace in the first equality, we have 
$$
\sigma(\epsilon)((1/2)^{n-1}+2(n-1))=\theta(\epsilon)
((1/2)^{n-1}+2(n-1))
$$ 
which implies $\sigma(\epsilon)=\theta(\epsilon)$. Taking trace in the second equality, we have $\lambda \sigma(\epsilon)=\theta(\epsilon)$ and hence we obtain that $\lambda=1$. Therefore, we have $\phi(C)=C_\sigma$ for all $C\in SL_n(\mathbb F)$.

After this, since in the real case $\sigma$ is necessarily the identity, the proof of Theorem \ref{specialreal} is complete.
As for the complex case, the necessity part of the proof of Theorem \ref{specialcomplex} is also done. Clearly, the sufficiency easily follows from Proposition \ref{endom}. 
\end{proof}

One can observe that the condition $n\geq 3$ was really needed in the proof above. Namely, we used it where we employed a version of the fundamental theorem of projective geometry and also where we selected diagonalizable matrices which are not similar to their inverses.

So, above we presented a common proof of the statements concerning the structures of the local automorphisms of the real and complex special linear groups. In the real case a much shorter argument can be given whose idea can also be employed in the proof of  Theorem \ref{generalreal}
on the local automorphisms of $GL_n(\mathbb R)$.

\begin{proof}[Separate short proof of Theorem \ref{specialreal}]
Consider the set $\mathcal B$ of all $n\times n$ real matrices that we list as follows. We consider the diagonal matrices with the entries $(1/2)^{n-1}$ in one place and 2 in the remaining $n-1$ places. Besides them, we also consider all matrices with the entries $(1/2)^{n-1}$ in the upper left corner, 2 in the other places of the diagonal and 1 in exactly one place off the diagonal and 0 elsewhere. The set $\mathcal B$ has $n^2$ elements all belonging to $SL_n(\mathbb R)$ and they can be easily seen to form a basis in $M_n(\mathbb R)$. We order those elements into the sequence $B_k$, $k=1,...,n^2$.

Let $\phi$ be a local automorphism of $SL_n(\mathbb R)$. We can see that either $\phi$ sends any of the elements of $\mathcal B$ to a similar matrix or any of the elements of $\mathcal B$ to a matrix that is similar to its transposed inverse. Clearly, we can assume without serious loss of generality that we have the former case (see the first part of the common proof of Theorems \ref{specialreal} and \ref{specialcomplex} above). It is then apparent that for any $A\in SL_n(\mathbb R)$ and $B\in \mathcal B$, we have an $S\in GL_n(\mathbb R)$ such that $\phi(A)=SAS^{-1}$, $\phi(B)=SBS^{-1}$ implying that $\tr(\phi(A)\phi(B))=\tr (AB)$.

We show that the elements $\phi(B_k)$, $k=1,...,n^2$ are linearly independent and therefore they form a basis in $M_n(\mathbb R)$. Assume that
$\sum_k \lambda_k \phi(B_k)=0$. We compute
$$
\begin{gathered}
0=\tr ((\sum_k \lambda_k \phi(B_k))\phi(B_j))=
\sum_k \lambda_k \tr (\phi(B_k)\phi(B_j)) \\
=\sum_k \lambda_k \tr (B_kB_j)=
\tr ((\sum_k \lambda_k B_k) B_j).
\end{gathered}
$$
Since this holds for all $B_j\in \mathcal B$ which form a basis in $M_n(\mathbb R)$, we have that $\tr ((\sum_k \lambda_k B_k)T)=0$ is valid for all $T\in M_n(\mathbb R)$. This implies that $\sum_k \lambda_k B_k=0$, so all $\lambda_k$-s are 0 and our claim follows.

After this, for any $\lambda_k\in \mathbb R$ and $A_k\in SL_n(\mathbb R)$, $k=1,...,m$ ($m$ is an arbitrary positive integer), we define
$$
\psi (\sum_k \lambda_k A_k) = \sum_k  \lambda_k \phi(A_k).
$$
We show that $\psi$ is well-defined. To see this, it is sufficient to check that $\sum_k \lambda_k A_k=0$ implies $\sum_k  \lambda_k \phi(A_k)=0$. Pick any $B_j\in \mathcal B$.
We have that
$$
\begin{gathered}
0=\tr ((\sum_k \lambda_k A_k)B_j)=\sum_k \lambda_k\tr(A_kB_j)\\
\sum_k \lambda_k\tr(\phi(A_k)\phi(B_j))
=\tr ((\sum_k \lambda_k\phi(A_k))\phi(B_j))
.
\end{gathered}
$$
Since the elements $\phi(B_j)$ form a basis in $M_n(\mathbb R)$, it follows as above that $\sum_k \lambda_k \phi(A_k) =0$. 

So, $\psi$ is well-defined and then we can infer that it is actually a bijective linear transformation on $M_n(\mathbb R)$ which extends $\phi$.

Since (by the local property of $\phi$), $\phi(I)=I$, we have $\psi(I)=I$. We know that $\phi(B_k)$ is similar to $B_k$, therefore, for any rank-one idempotent $P\in I_n^1(\mathbb R)$, pairing $(1/2)^{n-1}P+2(I-P)$ with any $B_k$,  there is $S\in GL_n(\mathbb R)$ such that
$$
\begin{gathered}
\psi((1/2)^{n-1}P+2(I-P))=\phi((1/2)^{n-1}P+2(I-P))\\=S((1/2)^{n-1}P+2(I-P))S^{-1}.
\end{gathered}
$$
Hence, by the linearity of $\psi$, we obtain that $\psi(P)=SPS^{-1}$.
Consequently, $\psi$ sends rank-one idempotents to rank-one idempotents.
By the main theorem in \cite{OmlSem}, we have the structure of all bijective linear transformations on $M_n(\mathbb R)$ which preserve rank-one idempotents. Namely, we have that there is $T\in GL_n(\mathbb R)$ such that $\psi$ is either of the form
$$
\psi(A)=TAT^{-1}, \quad A\in M_n(\mathbb R)
$$
or of the form 
$$
\psi(A)=TA^t T^{-1}, \quad A\in M_n(\mathbb R)
$$
and the same description applies for the restriction of $\psi$ onto $SL_n(\mathbb R)$, that is, for $\phi$. 

It remains to rule out the appearance of the second form. This can be done as follows. Assume on the contrary that we have $\psi(A)=TA^t T^{-1}$, $A \in M_n(\mathbb R)$ for a given $T\in GL_n(\mathbb R)$.
By the local property of $\phi$,
we know that any $\phi(A)$ is similar to $A$ (see above).
Consider the matrix
$$
A= \left[  \begin{matrix}  1 & 1 & 0 & \ldots & 0 \cr    0 & 1 & 0 & \ldots & 0 \cr   0 & 0 & 1 & \ldots & 0 \cr  \vdots & \vdots & \vdots & \ddots & \vdots \cr   0 & 0 & 0 & \ldots & 1 \cr    \end{matrix}                        \right]
$$
and a diagonal matrix $D\in SL_n(\mathbb R)$ with pairwise different diagonal entries which is not similar to its inverse. 
Then we have an $S\in GL_n(\mathbb R)$ such that
$$
TA^tT^{-1}=\phi(A)=SAS^{-1}, \quad TDT^{-1}=\phi(D)=SDS^{-1}.
$$
From the second equality we deduce that $T^{-1}S$ is diagonal and obtain that $A^t=T^{-1}S A(T^{-1}S)^{-1}$. A simple computation leads to a contradiction.

Therefore, we necessarily have 
$$
\phi(A)=TAT^{-1}, \quad A\in SL_n(\mathbb R)
$$
and the proof is complete.
\end{proof}

In the next result, we give a precise description of the local automorphisms of $GL_n(\mathbb R)$. We first introduce the following collections of functions.
Recall that in \eqref{E:functions} we defined $\Mre$ and $\Mrk$ as the sets of all functions $g:\mathbb R^\ast \to \mathbb R^\ast$ that appear in the descriptions \eqref{jao}, \eqref{jaojao} of the automorphisms of $GL_n(\mathbb R)$.

We now set
$$
\LMre=\{ f:\mathbb R^\ast \to \mathbb R^\ast\, |\, \forall t,s\in \mathbb R^\ast\, : \exists g\in \Mre : f(t)=g(t), f(s)=g(s)\},
$$
$$
\LMrk=\{ f:\mathbb R^\ast \to \mathbb R^\ast  \, |\, \forall t,s\in \mathbb R^\ast\, : \exists g\in \Mrk :\, f(t)=g(t), f(s)=g(s)\}.
$$
Hence, the elements of $\LMre$ are the functions on $\mathbb R^\ast$ which can be interpolated at any two points in $\mathbb R^\ast$ by elements of $\Mre$. Clearly, similar holds concerning $\LMrk$ and $\Mrk$.

After this, the description of the local automorphisms of $GL_n (\mathbb{R})$ is as follows.

\begin{theorem}\label{generalreal}
For a map $\phi : GL_n (\mathbb{R}) \to GL_n (\mathbb{R})$ the following assertions are equivalent.
\begin{itemize}
\item $\phi$ is a local automorphism. 
\item  Either there exist $T \in GL_n (\mathbb{R})$ and a function $f\in \LMre$ such that
$$
\phi (A) =   f(\det A) TA T^{-1} , \ \ \ A \in  GL_n (\mathbb{R}),
$$
or there exist $T \in GL_n (\mathbb{R})$ and a function $f\in \LMrk$ such that
$$
\phi (A) = f(\det A)  T(A^{-1})^t T^{-1} , \ \ \ A \in  GL_n (\mathbb{R}).
$$
\end{itemize}
\end{theorem}

\begin{proof}
Let $\phi : GL_n (\mathbb{R}) \to GL_n (\mathbb{R})$ be a local automorphism. Then, by Proposition \ref{P:GLaut}, we clearly have that
$$
\phi ( SL_n (\mathbb{R})) \subset  SL_n (\mathbb{R}) 
$$
and the restriction of $\phi$ to $SL_n (\mathbb{R})$, considered as a map from $SL_n (\mathbb{R})$ to itslef, is a local automorphism. By  
Theorem \ref{specialreal}, there exists a $T\in GL_n(\mathbb R)$ such that either
$\phi (A) =   TA T^{-1}$ for every $A \in  SL_n (\mathbb{R})$, or $\phi (A) =  T(A^{-1})^t T^{-1}$ for every $A \in  SL_n (\mathbb{R})$.
After replacing $\phi$ by $A \mapsto T^{-1}\phi (A) T$ in the first case, and by 
$A \mapsto T^{-1}\phi ((A^{-1})^t) T$ in the second case, we may assume with no loss of generality that $\phi (A) = A$  for every $A \in  SL_n (\mathbb{R})$. The rest of the proof of this part is to show that then $\phi$ is the identity on the whole set $GL_n (\mathbb{R})$.

In the first step, we show that there is a function $f:\mathbb R^\ast \to \mathbb R^\ast$ such that $\phi(\lambda A)=f(\lambda )A$ holds for all $\lambda \in \mathbb R^\ast$ and $A\in SL_n(\mathbb R)$.
To verify this, first observe that for each $\lambda\in \mathbb R^\ast$ and each $A \in  SL_n (\mathbb{R})$, there exist a matrix $S\in GL_n (\mathbb{R})$ and a multiplicative function $g : \mathbb{R}^\ast \to  \mathbb{R}^\ast$ (with certain additional properties) such that either 
$$
A =  \phi (A) = g(\det A) SA S^{-1}  = SA S^{-1} 
$$
and
$$
\phi (\lambda A) = g(\det (\lambda A)) S (\lambda A) S^{-1} = g(\lambda)^n \lambda SA S^{-1}= g(\lambda)^n \lambda A,
$$
or
$$
A =  \phi (A) = g(\det A) S(A^{-1})^t S^{-1} = S(A^{-1})^t S^{-1}
$$
and
$$
\phi (\lambda A) = g(\det (\lambda A))  S((\lambda A)^{-1})^t S^{-1} = g(\lambda)^n \lambda^{-1} S(A^{-1})^t S^{-1}=g(\lambda)^n \lambda^{-1} A.
$$
Hence, for every $A \in  SL_n (\mathbb{R})$, there exists a function $h_A :  \mathbb{R}^\ast \to  \mathbb{R}^\ast$ such that
$$
\phi (\lambda A) = h_A (\lambda ) A, \ \ \ \lambda \in \mathbb{R}^\ast.
$$

We show that $h_A = h_I = h$ is independent of $A$. 
Select $A\in SL_n(\mathbb R)$ and assume first that $A$ is not similar  to $-A$ and also $(A^{-1})^t$ is not similar to $-A$. We know that for 
each $\lambda\in \mathbb R^\ast$ there exist  $S\in GL_n(\mathbb R)$ 
and a multiplicative function $g : \mathbb{R}^\ast \to  \mathbb{R}^\ast$ such that either 
$$
h_I (\lambda )I =  \phi (\lambda I) = g(\det (\lambda I )) S (\lambda I) S^{-1}  = g(\lambda)^n \lambda I
$$
and
$$
h_A (\lambda ) A = \phi (\lambda A) = g(\det (\lambda A)) S (\lambda A) S^{-1} = g(\lambda)^n \lambda SA S^{-1},
$$
or
$$
h_I (\lambda )I =  \phi (\lambda I) = g(\det (\lambda I )) S ((\lambda I)^{-1})^t S^{-1}  = g(\lambda)^n \lambda^{-1} I
$$
and
$$
h_A (\lambda ) A =\phi (\lambda A) = g(\det (\lambda A))  S((\lambda A)^{-1})^t S^{-1} = g(\lambda)^n \lambda^{-1} S(A^{-1})^t S^{-1}.
$$
In both cases, we claim that our assumption on $A$ yields that $h_A (\lambda) = h_I (\lambda)$. Indeed, in the first case we have $h_A (\lambda ) A=h_I(\lambda) SA S^{-1}$. Taking determinants, we have that $h_A (\lambda )=\pm h_I(\lambda)$ but the negative sign here is excluded due to the assumption that $A$ is not similar to $-A$. The second case can be treated in the same manner. So, in our particular case we have $h_A=h_I$. (The reason for taking the determinant above and not trace is that the trace of the considered matrix $A$ may be zero.)

Let now $A \in  SL_n (\mathbb{R})$ be any matrix. We assert that we can find $B \in  SL_n (\mathbb{R})$ such that  $B$ is not similar  to $-B$, $(B^{-1})^t$ is not similar to $B$, $(B^{-1})^t$ is not similar to $-B$, and  for every invertible matrix $S$ which commutes with $B$ we have $SAS^{-1} \not= -A$. 
How to find such a matrix $B$? The first three conditions can easily be satisfied, any elements of $\mathcal E$ (see the common proof of Theorems \ref{specialreal}, \ref{specialcomplex}) would be appropriate. The fourth condition is automatically fulfilled if $\tr (A)\neq 0$. Indeed, taking traces in $SAS^{-1}= -A$, we would have $\tr(A)=0$. Assume now that $\tr A=0$. Since $A\neq 0$, we can select a rank-one idempotent $P\in \Pne$ such that $\tr (PA)\neq 0$. Consider $B=(1/2)^{n-1}P+2(I-P)$. If there exists $S\in GL_n(\mathbb R)$ such that $SB=BS$ and $SAS^{-1}=-A$, then it follows that
$S(BA)S^{-1}=-BA$. Taking trace, we have 
$\tr (BA)=0$ from which, using $\tr A=0$ and the form of $B$, we conclude that $\tr(PA)=0$, a contradiction. So, for any $S\in GL_n(\mathbb R)$, if $SB=BS$, then we necessarily have that $SAS^{-1}\neq -A$. This means that a required element $B \in  SL_n (\mathbb{R})$ can really be found. We know that $h_B=h_I$.

Now, select any number $\lambda \in \mathbb{R}^\ast$. Then there exist $S\in GL_n(\mathbb R)$ and a multiplicative function $g : \mathbb{R}^\ast \to  \mathbb{R}^\ast$ such that   
$$
h_I (\lambda )B = h_B (\lambda )B =\phi (\lambda B) = g(\det (\lambda B )) S (\lambda B) S^{-1}  = g(\lambda)^n \lambda S B S^{-1}
$$
and
$$
h_A (\lambda ) A = \phi (\lambda A) = g(\det (\lambda A)) S (\lambda A) S^{-1} = g(\lambda)^n \lambda SA S^{-1}
$$
(the possibility where transposed inverses appear cannot show up because of the choice of $B$).
By the properties of $B$, from the first equality we necessarily obtain that $SBS^{-1}=B$.
From the second equality, it follows that $SAS^{-1} = \pm A$ but, as $SB=BS$, we have $SAS^{-1} \not= -A$, so we conclude that $h_A (\lambda) =g(\lambda)^n \lambda= h_I (\lambda)$, as desired.
Thus, denoting $h_I$ by $h$, we have proved that 
\begin{equation}\label{E:egyf}
\phi (\lambda A) = h (\lambda ) A, \ \ \ \lambda \in \mathbb{R}^\ast, A \in  SL_n (\mathbb{R}).
\end{equation}

We next study how $\phi$ behaves on the set $SL_n^-(\mathbb R)$ of all elements of $GL_n(\mathbb R)$ with determinant $-1$. In the short proof of Theorem \ref{specialreal} above we considered a basis $\mathcal B$ of $M_n(\mathbb R)$ consisting of certain elements of $SL_n(\mathbb R)$. In all those matrices, alter the element $(1/2)^{n-1}$ to its negative. We denote the so obtained set by $\mathcal B'$ and its elements by $B'_k$, $k=1,...,n^2$. One can check that $\mathcal B'$ is also a basis in $M_n(\mathbb R)$ which consists of elements of $SL_n^-(\mathbb R)$. 

Select any $N\in SL_n^-(\mathbb R)$. Let us pair $N$ with any $B_k$. Using the local property of $\phi$ and the fact that $\phi$ is the identity on $SL_n(\mathbb R)$, we see that $\phi(N)$ is similar to $N$ or to $-N$ (this latter possibility may appear only in the case where $n$ is even since, for a multiplicative function $g$ that appears in the forms of automorphisms of $GL_n(\mathbb R)$, by the equality $g(-1)^n(-1)=-1$ (the automorphisms of $\mathbb R^\ast$ sends $-1$ to $-1$), we can have $g(-1)=-1$ only if $n$ is even). Observe that here $N$ is an arbitrary element of $SL_n^-(\mathbb R)$, so, in particular, for each $B_k'$, $\phi(B_k')$ is similar 
to $B_k'$ or to $-B_k'$.

Now, pair $N$ with any $B_k'$. We see that there is $c\in\{-1,1\}$ such that
\begin{equation}\label{E:M3}
\phi(N)=cSN S^{-1}, \quad \phi(B_k')=cSB_k' S^{-1}
\end{equation}
hold for some $S\in GL_n(\mathbb R)$
(the possibility involving transposed inverses cannot occur because, as we know, $\phi(B_k')$ is similar either to $B_k'$ or to $-B_k'$). Multiplying the equalities in \eqref{E:M3} and taking trace, we obtain that
\begin{equation}\label{E:M2}
\tr (\phi(N)\phi(B_k'))=\tr (NB_k').
\end{equation}
After this, just as in the short proof of Theorem \ref{specialreal}, one can check that the elements $\phi(B'_j)$, $j=1,...,n^2$ form a basis in $M_n(\mathbb R)$ and the transformation
$$
\psi: \sum_k \lambda_k N_k \longmapsto \sum_k\lambda_k \phi(N_k), \quad \lambda_k\in \mathbb R, \, N_k\in SL^-_n(\mathbb R)
$$
is a well-defined bijective linear map on $M_n(\mathbb R)$.
By \eqref{E:M2}, $\psi$ preserves the trace of products.

Pairing the elements $B_k', B_j'$, by the local property of $\phi$, we see that either we have 
$$\tr (\phi(B_k'))=\tr (B_k'), \quad 
\tr (\phi(B_j'))=\tr (B_j'),$$ or we have $$\tr (\phi(B_k'))=-\tr (B_k'), \quad 
\tr (\phi(B_j'))=-\tr (B_j').$$ 
It follows immediately that either 

(a) we have $\tr (\phi(B_j'))=\tr (B_j')$ for all $j=1,...,n^2$, 

\noindent
or 

(b) we have $\tr (\phi(B_j'))=-\tr (B_j')$ for all $j=1,...,n^2$. 

\noindent
Suppose that we have the latter case (b). By \eqref{E:M3}, for any $N\in SL_n^-(\mathbb R)$, $\phi(N)$ is similar to $-N$. Moreover, by the definition of $\psi$, we deduce that $-\psi$ preserves the trace.
We compute 
$$
\tr (\psi(A)\psi(I))=\tr (AI)=\tr (A)=-\tr (\psi(A))=\tr (\psi(A)(-I)), \quad A\in M_n(\mathbb R)
$$
which, by the surjectivity of $\psi$
implies that $\psi(I)=-I$. 

Now, choosing any rank-one idempotent $P\in \Pne$, we deduce that for some $S\in GL_n(\mathbb R)$ we have
$$
\begin{gathered}
\psi( -(1/2)^{n-1}P+2(I-P))=\phi( -(1/2)^{n-1}P+2(I-P))\\=-S(-(1/2)^{n-1}P+2(I-P))S^{-1}.
\end{gathered}
$$ 
By the linearity of $\psi$, we easily obtain from this equality that $\psi(P)=-SPS^{-1}$. Therefore, the linear bijection $-\psi$ preserves rank-one idempotents. In the same way as in the short proof of Theorem \ref{specialreal}, we get to the following:
there is a $T\in GL_n(\mathbb R)$
such that
$\phi$ is either of the form
$$
\phi(N)=-TNT^{-1}, \quad N\in SL_n^-(\mathbb R)
$$
or it is of the form 
$$
\phi(N)=-TN^t T^{-1}, \quad N\in SL_n^-(\mathbb R).
$$
We can rule out the second possibility in a way very similar to what we followed in the short proof of Theorem \ref{specialreal}. Let
$$
A= \left[  \begin{matrix}  -1 & 1 & 0 & \ldots & 0 \cr    0 & 1 & 0 & \ldots & 0 \cr   0 & 0 & 1 & \ldots & 0 \cr  \vdots & \vdots & \vdots & \ddots & \vdots \cr   0 & 0 & 0 & \ldots & 1 \cr    \end{matrix}                        \right]
$$
and pick a diagonal $D\in SL_n^-(\mathbb R)$ whose diagonal entries are pairwise different and which is not similar to $\pm D^{-1}$. Pairing $A$ and $D$, by the local property of $\phi$, we have $S\in GL_n(\mathbb R)$ such that
$$
-TDT^{-1}=\phi(D)=-SDS^{-1}, \quad
-TA^tT^{-1}=\phi(A)=-SAS^{-1}.
$$
It follows that for $R=T^{-1}S$ we both have
$D=RDR^{-1}$ and $A^t=RAR^{-1}$. But it is untenable. Indeed, by the first equality, $R$ is diagonal and one can trivially compute that with such a matrix the equality $A^t=RAR^{-1}$ cannot hold.
Therefore, we have that $\phi(N)=-TNT^{-1}$ for all $N\in SL_n^-(\mathbb R)$.

Going further, we observe that, since for every $N\in SL_n^-(\mathbb R)$ the determinant of $N^2$ is 1 and $\phi$ respects the square operation (which immediately follows from the local property of $\phi$), we  compute
$$ 
TN^2T^{-1}=\phi(N)^2=\phi(N^2)=N^2.
$$ 
As this holds for all $N\in SL_n^-(\mathbb R)$, we easily conclude that $T$ is scalar multiple of the identity and hence we have
$\phi(N)=-N$ for all $N\in SL_n^-(\mathbb R)$.

Pairing any $N\in SL_n^-(\mathbb R)$ with $\lambda N$ for any $\lambda>0$, we have two options. First, there is $S\in GL_n(\mathbb R)$ and $g\in \Mre$ such that
$$
-N=\phi(N)=g(-1)SNS^{-1}, \quad \phi(\lambda N)=g(-\lambda ^n)\lambda SNS^{-1}
$$
from which we can deduce that $\phi(\lambda N)=-\alpha(\lambda)N$ for some automorphism $\alpha$ of $\mathbb R^\ast$, or, second, there is $S\in GL_n(\mathbb R)$ and $g\in \Mrk$ such that
$$
-N=\phi(N)=g(-1)S{(N^{-1})}^t S^{-1}, \quad \phi(\lambda N)=g(-\lambda^n)\lambda^{-1} S{(N^{-1})}^tS^{-1}
$$
from which we again obtain that $\phi(\lambda N)=-\alpha(\lambda )N$ holds for some automorphism $\alpha$ of $\mathbb R^\ast$. Notice that $\alpha(\lambda)$ is necessarily a positive number.

Finally, since
$$
h(\lambda)^2I=\phi(\lambda I)^2=\phi(\lambda^2 I)=h(\lambda^2)I
$$
(in particular, $h$ is positive on the positive real numbers), we have
$$
\alpha(\lambda)^2 N^2=\phi(\lambda N)^2=\phi(\lambda^2 N^2)=h(\lambda^2)N^2=h(\lambda)^2N^2.
$$
This implies that $h(\lambda)=\alpha(\lambda)$ and,
therefore, we have 
$$
\phi(\lambda N)=-h(\lambda )N, \quad \lambda>0, N\in SL_n^-(\mathbb R). 
$$

In a similar fashion, if we have the case (a) above, i.e., if $\tr (\phi(B_j'))=\tr (B_j')$ for all $j=1,...,n^2$, then we obtain that $\phi(\lambda N)=h(\lambda )N$ holds for all positive $\lambda$ and $N\in SL_n^-(\mathbb R)$. 

After this, we can now describe the action of $\phi$ on the full set $GL_n(\mathbb R)$ as follows.
Any $B\in GL_n(\mathbb R)$ can be written as
$$
B= \sqrt[n]{|\det B|} \biggl(\frac{1}{\sqrt[n]{|\det B|}} B\biggr). 
$$ 
In the case (a), we can see that
$$
\phi(B)=h(\sqrt[n]{|\det B|})\biggl(\frac{1}{\sqrt[n]{|\det B|}} B\biggr)
$$
holds
independently of whether $\det B$ is positive or negative.
In the case (b) the situation is different. For $\det B>0$, we have
$$
\phi(B)=h(\sqrt[n]{|\det B|})\biggl(\frac{1}{\sqrt[n]{|\det B|}} B\biggr)
$$
while if $\det B<0$, then we have
$$
\phi(B)=-h(\sqrt[n]{|\det B|})\biggl(\frac{1}{\sqrt[n]{|\det B|}} B\biggr).
$$
However, 
the common feature of those two cases is that we have a function $f:\mathbb R^\ast \to \mathbb R^\ast$ such that $\phi(B)=f(\det B)B$ holds for all $B\in  GL_n(\mathbb R)$. It only remains to study this function $f$.

Pick arbitrary nonzero real numbers $\lambda,\mu$. Let $p=\sqrt[n]{|\lambda|}$, $q=\sqrt[n]{|\mu|}$. Choose (for example, diagonal) matrices $A\in SL_n(\mathbb R)$ and $N\in SL_n^-(\mathbb R)$ with nonzero traces such that no real multiple of $A$ is similar to the transposed inverse of $A$ and no real multiple of $N$ is similar to the transposed inverse of $N$.
In the case where $\lambda,\mu$ are both positive, we pair $pA$ and $qA$. If $\lambda,\mu$ are both negative, we pair $pN$ and $qN$. If $\lambda$ is positive and $\mu$ is negative, we pair $pA$ and $qN$. For example, in this latter case there are $S\in GL_n(\mathbb R)$ and $g\in \Mre$ such that
$$
f(p^n)pA=\phi(pA)=g(p^n)pSAS^{-1},\quad
f(-q^n)qN=\phi(qN)=g(-q^n)qSNS^{-1}.
$$ 
Taking traces, we have $f(\lambda)=g(\lambda)$, $f(\mu)=g(\mu)$.
This way we obtain that $f\in \LMre$. The more difficult part of the equivalence statement in the theorem is hence proved. By Proposition \ref{P:GLaut}, the converse statement is just trivial.

Let us make the final comment that in the case where $n$ is odd, actually we do not need to examine the local automorphisms of $GL_n(\mathbb R)$ on $SL_n^-(\mathbb R)$.
Indeed, in that case, $\mathbb R SL_n(\mathbb R)=GL_n(\mathbb R)$ holds, for any $B\in GL_n(\mathbb R)$ we have
$$
B=\sqrt[n]{\det B}\biggl(\frac{1}{\sqrt[n]{\det B}} B\biggr)
$$
so, by \eqref{E:egyf}, we deduce
$$
\phi(B)=h(\sqrt[n]{\det B})\biggl(\frac{1}{\sqrt[n]{\det B}} B\biggr).
$$
After this, defining $f(\lambda)=h(\sqrt[n]{\lambda})/ \sqrt[n]{\lambda}$, $\lambda \in \mathbb R^\ast$, we get
$$
\phi(B)=f(\det B)B,  \quad B\in GL_n(\mathbb R).
$$
The fact that $f\in \LMre$ can also be easily shown (recall the previous paragraph).
\end{proof}

Above we obtained the complete description of the local automorphisms of $GL_n(\mathbb R)$. It is natural to ask if some more information can be given about the real functions that appear in the formulation of the theorem, i.e., about the elements of $\LMre$ and $\LMrk$. The answer is affirmative. Let us consider only the case of $\LMre$.

First, we characterize the local automorphisms of the group $\mathbb  R^\ast$.
To do that, we
define an equivalence relation on the positive real half-line. For any $\lambda, \mu\in (0,\infty)$, we write $\lambda \sim \mu$ if and only if there exists a nonzero rational number $q$ such that $\lambda = \mu^q$. Let $k : (0, \infty) \to  (0, \infty)$ be any function satisfying the following condition. 

\begin{itemize}
\item[(P)]
We have $k(1)=1$ and 
for every pair of positive real numbers $\lambda , \mu$ we have $\lambda \sim \mu \iff k(\lambda) \sim k (\mu)$, and if $\lambda = \mu^q$ for some $q \in \mathbb{Q}^\ast = \mathbb{Q} \setminus \{ 0 \}$ then $k(\lambda) = k(\mu)^q$. 
\end{itemize}

It is easy to describe all functions having the property (P). Indeed, let $\Gamma$ be the set of equivalence classes of $(0,\infty)$ with respect to $\sim$ and $\varphi : \Gamma \to \Gamma$ an injective function satisfying $\varphi ( \{ 1 \} ) = \{ 1 \}$. For every $\gamma \in \Gamma$, $\gamma \not= \{ 1\}$, choose representatives $\lambda_\gamma \in \gamma$ and $\mu_\gamma \in \varphi (\gamma)$. We define $k : (0, \infty) \to (0, \infty)$ in the following way. We first set $k(1)=1$. Next, for each positive real number $\lambda \not=1$ there exists a unique $\gamma\in \Gamma$ such that $\lambda \in \gamma$. By the definition of the equivalence relation $\sim$, there is a unique $q \in \mathbb{Q}$ such that $\lambda = \lambda_{\gamma}^q$. We set 
$$
k( \lambda) =\mu_{\gamma}^q.
$$
It is easy to verify that $k$ has the property (P) and that each function having that property can be obtained in the above way. 

Let us say that a function $h : \mathbb{R}^\ast \to  \mathbb{R}^\ast$ has property (LAR) if it satisfies the following
\begin{itemize}
\item[(LAR)]
$h$ maps $(0,\infty)$ to itself, the restriction of $h$ to $(0, \infty)$
has property (P) and $h(-\lambda) = - h(\lambda)$ holds for every nonzero real $\lambda$. 
\end{itemize}

The next statement tells that the local automorphisms of $\mathbb R^\ast$ are exactly the functions on $\mathbb R^\ast$ that have property (LAR).

\begin{proposition}\label{scareal}
Let $h : \mathbb{R}^\ast \to  \mathbb{R}^\ast$ be a function. Then the following statements are equivalent.
\begin{itemize}
\item The function $h$ is a local automorphism of $\mathbb R^\ast$,
\item The function $h$ has property (LAR).
\end{itemize}
\end{proposition}

\begin{proof} 
We already know that each multiplicative function $f : \mathbb{R}^\ast \to  \mathbb{R}^\ast$ maps the set of all positive real numbers to itself.
Moreover, if $f$ is a bijection of $ \mathbb{R}^\ast$ onto itself, then it maps the set of all positive real numbers bijectively onto itself and for every pair of positive real numbers $\lambda , \mu$ we have $\lambda \sim \mu \iff f(\lambda) \sim f (\mu)$, and if $\lambda = \mu^q$ 
for some $q\in \mathbb{Q}^\ast$
then $f(\lambda) = f(\mu)^q$. Further, the bijectivity of $f$ yields $f(-1) = -1$, and therefore, $f(-\lambda) = - f(\lambda)$ holds for every nonzero real $\lambda$. It follows that the local automorphisms of $\mathbb R^\ast$ have the property (LAR).

To prove the converse we take any function $h$ with the property (LAR) and select a pair of real numbers $\lambda , \mu$. Assume first that both $\lambda$ and $\mu$ are positive. If $\lambda \sim \mu$, that is, $\lambda = \mu^q$ for some nonzero rational number $q$, then there exists a bijective additive function $a : \mathbb{R} \to \mathbb{R}$ such that 
$$
a( \log \mu ) = \log h (\mu).
$$
Indeed, this is trivial in the case when $\mu=1$. If $\mu\not=1$, then also $h (\mu) \not=1$, and the existence of a bijective $a$ satisfying the above equality is again obvious.
The function $f : \mathbb{R}^\ast \to  \mathbb{R}^\ast$ defined by $$f( \xi ) = 
e^{ a ( \log \xi) }$$
for $\xi \in (0, \infty)$ and $f( \xi) = - f(-\xi)$, $\xi \in (-\infty , 0)$, is clearly bijective and multiplicative and $f(\mu ) = h (\mu)$ and $f(\lambda) = f(\mu^q) = (f(\mu))^q = (h(\mu))^q=h(\lambda)$.

In the case where $\lambda \not\sim \mu$, we need to distinguish two cases. The first one is that $\lambda \not= 1$ and $\mu \not= 1$.
Then $\log \lambda$ and $\log \mu$ are $\mathbb{Q}$-linearly independent. Therefore we can find a bijective additive function $a : \mathbb{R} \to \mathbb{R}$ such that 
$$
a( \log \lambda ) = \log h (\lambda) \ \ \ {\rm and} \ \ \
a( \log \mu ) = \log h (\mu) .
$$
The function $f : \mathbb{R}^\ast \to  \mathbb{R}^\ast$ defined by $$f( \xi ) = 
e^{ a ( \log \xi) }$$
for $\xi \in (0, \infty)$ and $f( \xi) = - f(-\xi)$, $\xi \in (-\infty , 0)$, is clearly bijective and multiplicative and we have $f(\mu ) = h (\mu)$ and $f(\lambda)  = h(\lambda)$. The case where $\lambda = 1$ or $\mu = 1$ is trivial.

We continue with the possibility that one of the real numbers $\lambda , \mu$ is positive and the other one is negative, say, $\lambda >0$ and $\mu < 0$. Then as before we can find 
 a bijective multiplicative function $f : \mathbb{R}^\ast \to  \mathbb{R}^\ast$ such that 
$h (\lambda ) = f (\lambda)$ and $h (|\mu |) = f (|\mu|)$. Since $h$ has property (LAR) and $f$ is bijective we have $h(\mu ) = - h( -\mu)$ and  $f(\mu ) = - f( -\mu)$, and therefore,  $h(\mu ) =  f( \mu)$. In the same way we can treat the case where both $\lambda$ and $\mu$ are negative.
\end{proof}

After this, the elements of the set $\LMre$ that appear in Theorem \ref{generalreal} describing the structure of the local automorphisms of $GL_n(\mathbb R)$ can be characterized as follows.

\begin{proposition}\label{P:lmr1}
Let $f:\mathbb R^\ast\to \mathbb R^\ast$. Then $f\in \LMre$ if and only if the following two conditions hold.
\begin{itemize}
\item[(i)]
The function $\lambda\mapsto f(\lambda)^n\lambda$ has property (LAR),
\item[(ii)]
$f((0,\infty))\subset (0,\infty)$, and either
$f((-\infty,0))\subset (0,\infty)$ or 
$f((-\infty,0))\subset (-\infty, 0)$.
\end{itemize}
(The latter possibility in (ii) appears only when $n$ is even.)
\end{proposition}

\begin{proof}
The necessity of the conditions $(i), (ii)$ follows from Proposition \ref{scareal} and Observation \ref{polakur}.

Suppose now that
$f:\mathbb R^\ast \to \mathbb R^\ast$ satisfies $(i)$ and $(ii)$.
Apparently,
for every automorphism $\alpha$ of $\mathbb R^\ast$ there is an element $g\in \Mre$ such that $\alpha(\lambda)=g(\lambda)^n\lambda$ for all $\lambda\in \mathbb R^\ast$.

Therefore, selecting any two elements $\lambda,\mu$ of $\mathbb R^\ast$, we have $g\in \Mre$ such that
$$
f(\lambda)^n\lambda= g(\lambda)^n\lambda,\quad
f(\mu)^n\mu= g(\mu)^n\mu.
$$
If $n$ is odd, we immediately obtain that $f(\lambda)=g(\lambda)$ and $f(\mu)=g(\mu)$. 

Assume that $n$ is even. If $\lambda,\mu$ are positive, then we again have $f(\lambda)=g(\lambda)$ and $f(\mu)=g(\mu)$. 
In the remaining cases, we use property $(ii)$ and the fact that, as one can easily see (or refer to Observation \ref{polakur}),  one can alter $g$ on $(-\infty ,0)$ to its negative and obtain a function which also belongs to the class $\Mre$. This way we can conclude that at any two points of $\mathbb R^\ast$, the function $f$ can be interpolated by some element of $\Mre$. This finishes the proof.
\end{proof}

So, we now have a fairly good picture about the local automorphisms of the real general linear group. One may suspect that the complex case is far more complicated. Indeed, we do not have the precise description of the local automorphisms of $GL_n(\mathbb C)$, we leave that question as an open problem. 

We 
next move to the complex special unitary group $SU_n(\mathbb C)$ and show that all of its local automorphisms are in fact automorphisms.

\begin{theorem}\label{T:SU}
Every local automorphism of $SU_n(\mathbb C)$ is an automorphism.
\end{theorem}

\begin{proof}
Denote by $\Pn$ the set of all rank-one projections (Hermitian idempotents) in $M_n(\mathbb C)$.
Pick and fix complex numbers $\alpha,\beta\in \Se$ such that $\alpha^n\neq 1$, $\alpha\beta^{n-1}=1$, $\overline{\{\alpha,\beta\}}\neq \{\alpha,\beta\}$. Consider the set $\Eu$ of all elements of $SU_n(\mathbb C)$ of the form
$$
\alpha P+\beta (I-P)
$$
where $P\in \Pn$. By the forms of the automorphisms of $SU_n(\mathbb C)$, see \eqref{F:su}, the image of such an element under $\phi$ is either of the form
$\alpha SPS^{-1} +\beta S(I-P)S^{-1}$
or of the form
$\overline{\alpha}S P^t S^{-1} +\overline{\beta} S(I-P^t)S^{-1}$ with some $S\in U_n(\mathbb C)$, depending on $P$.

By the local property of $\phi$, examining spectra, we see that 
either we always have the first form or we always have the second form for the elements of $\Eu$, the possibility of "mixing" is ruled out. 
Composing $\phi$ with the automorphism $A\mapsto \overline{A}$ if necessary (i.e., in the latter case above), we can and do assume that $\phi$ maps $\Eu$ into itself.

Therefore, we have a map $\xi:\Pn\to \Pn$ such that
$$
\phi(\alpha P+\beta (I-P))=\alpha \xi(P)+\beta(I-\xi(P))
$$
holds for all $P\in \Pn$. Observe that $\alpha\neq \beta$.
As we have already referred to it before, local automorphisms preserve commutativity and hence, for any $P,Q\in \Pn$ with $PQ=0$ we have the commutativity of $\xi(P)$ and $\xi(Q)$, which implies that either $\xi(P)=\xi(Q)$ or $\xi(P)\xi(Q)=0$. Since the first equality would imply that $P=Q$ which is untenable, it follows that 
$\xi(P)\xi(Q)=0$. Consequently, $\xi$ is a map on $\Pn$ which preserves orthogonality. The structure of such maps is known.
By Theorem 1 in \cite{Pank}, it follows that $\xi$ is a so-called Wigner transformation. This means that there is a $T\in U_n(\mathbb C)$ such that either
we have 
$$
\xi(P)=TPT^{-1}, \quad P\in \Pn
$$
or we have 
$$
\xi(P)=TP^t T^{-1}, \quad P\in \Pn.
$$ 
Composing $\phi$ with the automorphism $A\mapsto T^{-1}AT$, it follows that we can and do assume that $\xi$ is either the identity on $\Pn$ or the transposition on $\Pn$.

In the next part of the proof we show that this latter possibility cannot occur. To see this, assume that $\xi(P)=P^t$ for all $P\in \Pn$ which implies that $\phi$ is the transposition on $\Eu$. Pick any $A\in SU_n(\mathbb C)$. Choose an arbitrary rank-one projection $P\in \Pn$ and consider $E=\alpha P+\beta (I-P)$. By the local property of $\phi$, we have
$S\in U_n(\mathbb C)$ such that
$$
E^t=\phi(E)=SES^{-1}, \quad \phi(A)=SAS^{-1}.
$$
Multiplying these two equations and taking trace, we have 
$$
\tr (\phi(A)E^t)=\tr (AE)=\tr(AE)^t=\tr A^t E^t
$$
and, since $\tr(\phi(A))=\tr (A)=\tr (A^t)$, we can infer that $\tr (\phi(A)P^t)=\tr (A^t P^t$). 
Since $P\in \Pn$ is arbitrary, we obtain that $\phi(A)=A^t$. This means that $\phi$ is the transposition on the whole set $SU_n(\mathbb C)$. 

Consider now $n$ different elements $\epsilon_1,...,\epsilon_n$ of the unit circle $\Se$ such that their set is not invariant under the conjugation of the complex field and whose product is 1. Choose  a diagonal matrix $A$ in $SU_n(\mathbb C)$ with diagonal elements $\epsilon_1,...,\epsilon_n$.
Let $b_1,...,b_n$ be any orthonormal basis in $\mathbb C^n$ and $B$ the unitary with eigenvalues $\epsilon_1,...,\epsilon_n$ and corresponding eigenvectors $b_1,...,b_n$. By the local property of $\phi$, we have an $S\in U_n(\mathbb C)$ such that 
$$
A=A^t=\phi(A)=SAS^{-1},\quad B^t = \phi(B)=S BS^{-1}. 
$$
From the former equality we see that $S$ is necessarily diagonal while from the latter equality we can easily deduce that
for some scalars $\gamma_i$ of modulus 1 we have 
\begin{equation}\label{E:M1}
Sb_i =\gamma_i \overline{b_i}
\end{equation}
Indeed, from $Bb_i=\epsilon_i b_i$, dividing by $\epsilon_i$ and taking conjugates and then multiplying by $\overline{B}^{-1}=B^t$, we have  $B^t\overline b_i=\epsilon_i \overline b_i$. We know that 
$B^tS=SB$ which implies $B^tSb_i=S\epsilon _i b_i=\epsilon_i Sb_i$ and then we deduce \eqref{E:M1}.

So, for any orthonormal basis $b_1,...,b_n$ in $\mathbb C^n$
we have a diagonal $S\in U_n(\mathbb C)$ and scalars $\gamma_i$ of modulus 1 such that the equalities \eqref{E:M1} hold. But this is not true. Indeed, if 
$$
b_1=(1/\sqrt{n})  [1,...,1]^t,
$$
then from \eqref{E:M1} we necessarily have $S=\gamma_1 I$. Next, choosing  
$b_2$ as the normalization of 
$$
[1,i,-(1+i),0,...]^t,
$$
we trivially see that \eqref{E:M1} cannot hold for $i=2$.

Thus we obtain that the Wigner transformation $\xi$ above is necessarily the identity on $\Pn$.
Then $\phi(E)=E$ for all $E\in \Eu$ and we can easily complete the proof. Indeed, as above, pick any $A\in SU_n(\mathbb C)$. Choose an arbitrary rank-one projection $P\in \Pn$ and consider $E=\alpha P+\beta (I-P)$. By the local property of $\phi$, we have
$S\in U_n(\mathbb C)$ such that
$$
E=\phi(E)=SES^{-1}, \quad \phi(A)=SAS^{-1}.
$$
From this we have 
$
\tr (\phi(A)E)=\tr (AE)
$
and then we deduce 
$$
\tr (\phi(A)P)=\tr (A P)  
$$
for any $P\in \Pn$. This implies $\phi(A)=A$ for all $A\in SU_n(\mathbb C)$ and we are done.
\end{proof}

For the description of the local automorphisms of the unitary group $U_n(\mathbb C)$, we introduce a collection of functions on $\Se$.
Let us denote 
\begin{equation*}\label{E:Mu}
\begin{aligned}
\Mu=&\{g: \Se \to \Se\, |\, g \text{ is multiplicative}, f(\lambda) = g(\lambda)^n \lambda , \,\lambda \in \Se \\ &\text{ is an automorphism of } \Se\}.
\end{aligned}
\end{equation*}
Proposition \ref{P:autunitary} tells that the automorphisms of $U_n (\mathbb{C})$ are exactly the transformations $\phi: U_n (\mathbb{C}) \to 
U_n (\mathbb{C})$ which are either of the form
\begin{equation*}
\phi (A) = g(\det A) TA T^{-1} , \ \ \ A \in  U_n (\mathbb{C}),
\end{equation*}
or of the form
\begin{equation*}
\phi (A) = g(\det \overline{A}) T\overline{A} T^{-1} , \ \ \ A \in  U_n (\mathbb{C}).
\end{equation*}
where $g\in \Mu$ and $T\in U_n (\mathbb{C})$.

Let us next set
$$
\LMu=\{f:\Se\to \Se \,|\, \forall \lambda,\mu \in \Se :\, \exists g\in \Mu : f(\lambda)=g(\lambda), f(\mu)=g(\mu)\}.
$$ 

The structure of the local automorphisms of $U_n(\mathbb C)$ is given in the next result.

\begin{theorem}\label{unitary}
For a map $\phi : U_n (\mathbb{C}) \to U_n (\mathbb{C})$ the following assertions are equivalent.
\begin{itemize}
\item $\phi$ is a local automorphism. 
\item  There exist $T \in U_n (\mathbb{C})$ and a function $f\in \LMu$ such that either
$$
\phi (A) =   f(\det A) TA T^{-1} , \ \ \ A \in  U_n (\mathbb{C}),
$$
or
$$
\phi (A) =   f(\det \overline{A}) T\overline{A} T^{-1} , \ \ \ A \in  U_n (\mathbb{C}).
$$
\end{itemize}
\end{theorem}

\begin{proof}
Only one direction requires proof.
Let $\phi: U_n (\mathbb{C}) \to U_n (\mathbb{C})$ be a local automorphism. By Proposition \ref{P:autunitary}, the restriction of $\phi$ to $SU_n(\mathbb C)$ is a local automorphism. Therefore, by Theorem \ref{T:SU}, it is an automorphism. It easily follows that without serious loss of generality, we can assume that $\phi$ is the identity on $SU_n(\mathbb C)$.

We can see that for any $A\in SU_n(\mathbb C)$, we have a function $h_A:\Se\to \Se$ such that $\phi(\lambda A)=h_A(\lambda )A$. Indeed, for a given $\lambda \in \Se$, there is $g\in \Mu$ and $S\in U_n(\mathbb C)$ such that
either
$$
A=\phi(A)=SAS^{-1}, \quad \phi(\lambda A)=g(\lambda)^n\lambda SAS^{-1}
$$
or 
$$
A=\phi(A)=S\overline{A}S^{-1}, \quad \phi(\lambda A)=g({\overline{\lambda}})^n\overline{\lambda} S\overline{A}S^{-1}.
$$
In both cases we obtain that $\phi(\lambda A)$ is a scalar multiple of $A$, so we have a function $h:\Se \to \Se$, such that $\phi(\lambda A)=h_A(\lambda) A$, $\lambda \in \Se$.

Select $A\in SU_n(\mathbb C)$ with the following property: the only rotation of $\Se$ which maps the spectrum of $A$ onto itself is the trivial one (the identity) and there is no rotation of $\Se$ that would map the spectrum of $A$ onto its conjugate, i.e., its reflection over the real axis.
We show that $h_A=h_I$. Let $\lambda$ be any element of $\Se$. We have two possibilities. The first one is that we have $g\in \Mu$ and  $S\in U_n(\mathbb C)$ such that
$$
h_A(\lambda)A=\phi(\lambda A)=g(\lambda)^n\lambda SAS^{-1}
$$
and
$$
h_I(\lambda)I=\phi(\lambda I)=g(\lambda)^n\lambda SIS^{-1}.
$$
By the properties of $A$, examining the spectra, from the first equality we deduce that $h_A(\lambda)=g(\lambda)^n \lambda$, while from the second one we trivially get $h_I(\lambda)=g(\lambda)^n\lambda$. Therefore, $h_A(\lambda)=h_I(\lambda)$.

The second possibility is that we have $g\in \Mu$ and $S\in U_n(\mathbb C)$  such that
$$
h_A(\lambda)A=\phi(\lambda A)=g(\overline{\lambda})^n\overline{\lambda}S \overline{A}S^{-1}
$$
and
$$
h_I(\lambda)I=\phi(\lambda I)=g(\overline{\lambda})^n \overline{\lambda}S\overline{ I}S^{-1}.
$$
However, by the conditions on $A$, the first equality cannot hold.

Let now $B\in SU_n(\mathbb C)$ be an arbitrary element. Choose $A$ as above with the additional properties that it has $n$ distinct eigenvalues and it commutes with $B$.
Then, by the local property of $\phi$ and the special properties of $A$, we have a $g\in \Mu$ and
an $S\in U_n(\mathbb C)$ such that
$$
h_A(\lambda)A=\phi(\lambda A)=g(\lambda)^n \lambda SAS^{-1}
$$
and
$$
h_B(\lambda)B=\phi(\lambda B)=g(\lambda)^n\lambda SBS^{-1}.
$$
From the first equality we have $h_A(\lambda)=g(\lambda)^n\lambda$ and then that $A=SAS^{-1}$. Since $A$ has different eigenvalues, $S$ necessary takes a diagonal form with respect to the basis consisting of the eigenvectors of $A$ and, by the commutativity of $A$ and $B$, the same holds for $B$. It follows that $S$ and $B$ commute, so we have $B=SBS^{-1}$. Hence   it follows that $h_B(\lambda)=g(\lambda)^n\lambda=h_A(\lambda)$. We deduce that $h_B=h_A=h_I$.

So, there is a function $h:\Se\to \Se$ such that for any $A\in SU_n(\mathbb C)$ we have
$$
\phi(\lambda A)=h(\lambda )A, \quad \lambda\in \Se.
$$
Let us see what we can say concerning this function $h$.
Again, select $A\in SU_n(\mathbb C)$ as above, that is, with the following properties: the only rotation of $\Se$ which maps the spectrum of $A$ onto itself is the trivial one and there is no rotation of $\Se$ that would transform the spectrum of $A$ onto its reflection over the real axis.
Choose $\lambda,\mu\in \Se$. Then there is a $g\in \Mu$ and an $S\in U_n(\mathbb C)$ such that
$$
h(\lambda)A=\phi(\lambda A)=g(\lambda)^n\lambda SATS^{-1},\quad 
h(\mu)A=\phi(\mu A)=g(\mu)^n\mu SAS^{-1}.
$$
It follows that 
$$
h(\lambda)=g(\lambda)^n\lambda, \quad h(\mu)=g(\mu)^n\mu,
$$
consequently, the values $h(z)/z$ depend only on $z^n$, $z\in \Se$. So, we have a function $k:\Se \to \Se$ such that 
$$
h(z)/z=k(z^n), \quad z\in \Se
$$
and for any $\lambda', \mu'\in \Se$, there is a $g\in \Mu$  such that $k(\lambda')=g(\lambda'), k(\mu')=g(\mu')$. Therefore, we have $k\in \LMu$. 

Finally, for an arbitrary $A\in U_n(\mathbb C)$, and $\lambda\in \Se$ such that $\lambda^n =\det A$, we compute
$$
\phi(A)=\phi\biggl(\lambda\biggl(\frac{1}{\lambda} A\biggr)\biggr)=h(\lambda)\frac{1}{\lambda} A=k(\det A)A.
$$
This completes the proof.
\end{proof}

\section{Concluding remarks}

With the above results we try to attract the attention to an, in our opinion, interesting area of research, the study of local automorphisms of groups. As it apparently follows from the results above, there is a big room for further investigations even concerning the classical groups. We mean, just to start with, the complex general linear group, the special orthogonal group, the  orthogonal group, and then further the real and complex symplectic groups, etc. Of course, one may also consider scalar fields different from the real and the complex ones.

Above we presented statements saying that local automorphisms are global automorphisms only in two cases: for the real special linear group and for the special unitary group.   There are local automorphisms of $GL_n(\mathbb R)$ which are not automorphisms and hence we have a
natural question weather some additional conditions on the local automorphisms guarantee that they are in fact automorphisms. Referring back to the introduction and the paper \cite{LaS}, the first candidate for such an assumption is surjectivity (any local automorphism is injective). However, it is not difficult to create a surjective local automorphism of $GL_n(\mathbb R)$ which is not an automorphism.
To do that, we choose a function $l:\mathbb R\to \mathbb R$ with the following properties: $l(0)=0$ and, on each one-dimensional subspace of $\mathbb R$ over $\mathbb Q$, $l$ acts as a multiplication by some nonzero rational number which number depends on the one-dimensional subspace in question.
Clearly, $l$ is a bijection of $\mathbb R$ and we can and do choose it in a way that it differs from any additive bijection.
Now, define $h:\mathbb R^\ast \to \mathbb R^\ast$ as follows. For any positive real number $\lambda$, we set
$$
h(\lambda)=
\frac{e^{l(\log \lambda)}}{\sqrt[n]{\lambda}},
$$   
while for any negative $\lambda$ we set $h(\lambda)=h(-\lambda)$.  
We prove that $h\in \LMre$. By Propositions \ref{scareal} and \ref{P:lmr1}, we only need to see
that the function $k:\mathbb R^\ast \to \mathbb R^\ast$ defined as
$$
k(\lambda)=e^{n l(\log \lambda)}
$$
for positive $\lambda$
and $k(\lambda)=-k(-\lambda)$ for negative $\lambda$ is a local automorphism of $\mathbb R^\ast$. That can be checked rather easily considering any two nonzero real numbers $\lambda,\mu$ and distinguishing some cases depending on the signs of $\lambda,\mu$ and also on whether $\log |\lambda|,\log |\mu|$ are linearly dependent or linearly independent over $\mathbb Q$ (cf. the proof of Proposition \ref{scareal}).
It then follows by Theorem \ref{generalreal} that the  map $\phi(A)=h(\det A)A$, $A\in GL_n(\mathbb R)$ is a local automorphism of $GL_n(\mathbb R)$. However, by Proposition \ref{P:GLaut}, this $\phi$ is not an automorphism as $l$ differs from all additive bijections of $\mathbb R$.
It remains to see that $\phi$ is surjective which is proved if it is shown that for any $A\in GL_n(\mathbb R)$, the equation
$$
h(\lambda^n \det A)\lambda=1
$$ 
has a solution $\lambda\in \mathbb R^\ast$. But, by the surjectivity of $l$, this is rather easy to verify.

Besides surjectivity, another possible extra condition may be to bring in topology and consider the given groups as topological groups and study their topological automorphisms, that is, algebraic automorphisms which are also homeomorphisms. Sure, that means that we leave the area of pure algebra, but in some sense it looks a better choice than surjectivity. 
It is easy to describe the topological automorphisms of $GL_n(\mathbb R)$. Indeed, by Proposition \ref{P:GLaut}, one only needs to determine the continuous elements of $\Mre, \Mrk$.
Since the continuous additive bijections of $\mathbb R$ are the nonzero scalar multiples of the identity, by Observation \ref{polakur}, we can see that the continuous function $g:\mathbb R^\ast \to \mathbb R^\ast$ belongs to $\Mre$ if and only if the following hold. There is a real constant $c\neq -1/n$ such that $g(\lambda)=\lambda^c$, $\lambda>0$ and, as for the behaviour of $g$ on the real negative half-line, in the case where $n$ is odd, we have $g(\lambda)=g(-\lambda)$, $\lambda<0$, while in the case where $n$ is even we have either
$g(\lambda)=g(-\lambda)$, $\lambda<0$ or $g(\lambda)=-g(-\lambda)$, $\lambda<0$. In a similar fashion, one can characterize the continuous elements of $\Mrk$. After this, using Theorem \ref{generalreal}, it is not difficult to verify that every map on $GL_n(\mathbb R)$ which can be interpolated by  topological automorphisms at any two points of $GL_n(\mathbb R)$ (i.e., any local topological automorphism of $GL_n(\mathbb R)$) is necessarily a topological automorhism.

We can deduce a similar statement concerning the group of topological automorphisms of
$SL_n(\mathbb C)$ using the well-known fact that the only continuous nonzero field endomorphims of $\mathbb C$ are the identity and the conjugation.
Finally, since there are only two topological automorphisms of $\Se$ (namely the identity and the conjugation), it is not difficult to verify that every map on $U_n(\mathbb C)$ which can be interpolated by topological automorphisms of $U_n(\mathbb C)$ at any two points, is necessarily a topological automorphism.

\end{document}